\pdfoutput=1
\documentclass[a4paper]{article}

\usepackage{geometry}
\usepackage{amsthm}
\usepackage[numbers,sort]{natbib}
\usepackage{amsmath,amssymb,amsfonts}
\usepackage{algorithm}
\usepackage{algpseudocode}
\usepackage{booktabs}
\usepackage{subcaption}
\usepackage{enumitem}
\usepackage{multirow}
\usepackage[european, straightvoltages, straightlabels]{circuitikz}
\usepackage{tikz}
\usetikzlibrary{decorations.pathmorphing, positioning}
\usepackage{pgfplots}
\usepackage{pgfplotstable}
\pgfplotsset{compat=1.18}
\usepackage{graphicx}
\usepackage{bbm}
\usepackage{xcolor}
\usepackage{hyperref}

\usepackage{cancel}

\theoremstyle{plain}
\newtheorem{theorem}{Theorem}[section]
\newtheorem{lemma}[theorem]{Lemma}
\newtheorem{proposition}[theorem]{Proposition}
\newtheorem{corollary}[theorem]{Corollary}
\theoremstyle{definition}
\newtheorem{definition}[theorem]{Definition}
\newtheorem{assumption}[theorem]{Assumption}
\theoremstyle{remark}
\newtheorem{remark}[theorem]{Remark}

\newcommand{\diff}{\mathrm{d}}

\newcommand{\R}{\mathbbm{R}}

\renewcommand{\H}{\mathcal{H}}
\renewcommand{\O}{\mathcal{O}}

\definecolor{darkgreen}{rgb}{0.1,0.7,0.1}
\graphicspath{{./figures/}}

\begin{document}
%------------------------------------------------------------------

\title{A Generalized Scalar Auxiliary Variable Method for Structure-Preserving and Efficient Integration of Nonlinear port-Hamiltonian DAEs\thanks{This work was funded by the Deutsche Forschungsgemeinschaft
    (DFG, German Research Foundation) -- Project-ID 531152215 -- CRC 1701.}}

\author{Aashutosh Sharma\thanks{Department of Applied and Computational Mathematics,
  University of Wuppertal, Gau{\ss}stra{\ss}e 20, 42119 Wuppertal, Germany.
  \texttt{\{asharma,bartel\}@uni-wuppertal.de}}
  \and
  Andreas Bartel\footnotemark[2]%
  \and
  Manuel Schaller\thanks{Faculty of Mathematics, Chemnitz University of Technology,
  Stra{\ss}e der Nationen 62, 09111 Chemnitz, Germany.
  \texttt{manuel.schaller@math.tu-chemnitz.de}}}

\date{}

\maketitle

\begin{abstract}
	
		We develop an energy-optimal generalized scalar auxiliary
		variable (EOP-GSAV) 
		framework for nonlinear index-one
		port-Hamiltonian differential-algebraic equations (pH-DAEs). 
		Exploiting the
		port-Hamiltonian structure, we separate the nonlinear effort, interconnection
		and dissipation terms from a constant implicit core.
		The resulting
		BDF-1 and BDF-2 schemes require one linear solve per time
		step with a reusable factorization, while retaining discrete passivity
		and accurate 
		tracking of the Hamiltonian. The schemes are
		compared with the implicit midpoint method equipped with full, modified,
		and frozen-Jacobian Newton iterations. 
		Numerical experiments ranging from a strongly state-dependent nonlinear
		stress test to large-scale benchmarks demonstrate robust and competitive
		performance, with substantial efficiency gains in matched-accuracy
		regimes. 
		A comparison with SUNDIALS IDA shows comparable work--precision behavior
        at equal order despite a non-specialized Python/SciPy implementation,
        while unrestricted variable-order adaptive IDA is faster in the
        high-accuracy regime.
\end{abstract}

\section{Introduction}

	Port-Hamiltonian systems~\cite{maschke1992intrinsic,duindam2009modeling,VanDerSchaftJeltsema2014}
	provide a systematic framework for dynamical systems in which energy
	storage, interconnection, dissipation, and external power exchange are
	represented explicitly.
	Applications include electrical engineering~\cite{bartel2022port}, multibody systems~\cite{berger2025port},
	fluid--structure interaction~\cite{rashad2021port}, electromagnetism~\cite{Clemens2024}, robotics~\cite{duong2024port}, and continuum
	mechanics~\cite{rashad2025port}. Constrained formulations lead naturally
	to port-Hamiltonian differential-algebraic equations (pH-DAEs)~\cite{BeattieMehrmannXu2018}, and structure-preserving spatial
	discretizations of distributed port-Hamiltonian systems 	likewise produce finite-dimensional port-Hamiltonian ODEs or DAEs~\cite{jacob2012linear,kinon2023port,brugnoli2021structure,
		ponce2025constrained}. This makes the pH-DAE setting a natural finite-dimensional target for time integration of structure-preserving semi-discretizations of distributed-parameter systems. 
	    We will consider a class of nonlinear pH-DAEs in which the nonlinearity enters through the effort map, the interconnection and the dissipation, see \eqref{eq:nonlinear-pHDAE} below.

	Time discretization of DAEs, e.g., by a BDF method, leads in general to a nonlinear algebraic system in every time step. For the pH-DAEs considered here, the central numerical objective is therefore to retain the energetic structure without rendering these nonlinear solves prohibitively expensive.
	
	Discrete-gradient methods achieve preservation of the energy balance through an
	exact discrete chain rule~\cite{Gonz96,mehrmann2019structure,schulze2023structure}.
	Recently, such an approach was developed for nonlinear
	pH-DAEs~\cite{kinon2025discreteDAE}.
	Still, the resulting nonlinear systems may require repeated residual and Jacobian
	evaluations, matrix assembly, and factorizations. For large-scale
	problems, this nonlinear algebra can dominate the cost of the time
	integration.

	Auxiliary-variable methods address the tension between unconditional energy
	stability and computational efficiency. The scalar auxiliary variable
	(SAV) approach~\cite{Shen2018} and its generalized form GSAV~\cite{HuangShen2022} introduce a scalar variable based on a shifted
	energy, allowing nonlinear contributions to be treated explicitly while
	retaining a linearly implicit discretization. Because the resulting
	auxiliary energy need not coincide with the true energy, relaxed and
	energy-optimal corrections have been introduced, including RSAV~\cite{jiang2022improving} and EOP-SAV~\cite{LiuZhangLi2024}. Recent
	connections between SAV and discrete-gradient formulations further show
	that auxiliary-variable ideas extend beyond their original
	gradient-flow setting~\cite{celledoni2026unified}.
	However, a GSAV algorithm for pH-DAEs of type~\eqref{eq:nonlinear-pHDAE}
	is currently not available.
	Such a formulation must incorporate the external port, respect the energy-induced geometry, accommodate degenerate dissipation, and remain compatible with the descriptor structure.

\paragraph{Contribution.}
In this work, we develop an energy-optimal GSAV (EOP-GSAV) framework for
nonlinear index-one pH-DAEs~\eqref{eq:nonlinear-pHDAE}. Its main advantage is
that all state dependence, be it in the effort map, in the interconnection or
in the dissipation, is gathered in a single term which is evaluated explicitly,
such that every time step requires one linear solve with the constant matrix
$\frac{\alpha_k}{\delta t}E+A$, to be introduced in
Section~\ref{sec:gsav-to-phdae}. This matrix is invertible for every
$\delta t>0$ due to the pH structure alone, Proposition~\ref{prop:solvability},
such that a single factorization serves the whole integration, while the order
of the underlying BDF method and the energetic structure are retained.
In particular, the algebraic work per step does not depend on the strength of
the nonlinearity, whereas methods based on nonlinear implicit solves require
repeated residual and Jacobian evaluations, matrix assembly, and
factorizations. We stress that the scheme is of fixed order and uses a fixed
step size, such that a variable-order adaptive solver may well be faster in the
high-accuracy regime. We will still observe in
Section~\ref{sec:ida-comparison} that it remains competitive with SUNDIALS IDA
at equal order, which motivates higher-order and adaptive EOP-GSAV extensions
as future work.

This paper features two main contributions, the first being the discretization
itself, Section~\ref{sec:gsav-to-phdae}: the scalar auxiliary variable is
driven by the open power balance \eqref{eq:modified-energy-law} rather than by
a closed-system decay law, and the closed-system EOP bound
of~\cite{LiuZhangLi2024} is replaced by a passivity budget which accounts for
the energy supplied through the external port. This yields the discrete
passivity inequality of Lemma~\ref{lem:auxiliary-passivity} and exact recovery
of the Hamiltonian whenever the budget is inactive,
Corollary~\ref{cor:exact-energy-tracking}. The second contribution is a unified
convergence analysis of BDF-1 and BDF-2 in the index-one descriptor setting,
Section~\ref{sec:gsav-phdae-analysis}, yielding order-$k$ convergence of the
state together with consistency and positivity of the auxiliary variable,
Theorem~\ref{thm:convergence}.

\paragraph{Outline.}
Section~\ref{sec:gsav-to-phdae} introduces the pH model class, the
constant-core splitting, and the BDF-$k$ EOP-GSAV discretization.
Section~\ref{sec:gsav-phdae-analysis} establishes solvability, convergence,
energy fidelity, and discrete auxiliary passivity. The numerical part follows
the same line. In Section~\ref{sec:toy-stress}, we first isolate the robustness
effect on a small, strongly state-dependent problem for which Jacobian reuse
becomes difficult, before we turn to a large-scale FPU-$\beta$/Maxwell
benchmark, Section~\ref{sec:time-comparison}, which is much more favorable to
Newton-type methods. Therein, we compare with the implicit midpoint method
equipped with full, modified, and frozen-Jacobian Newton iterations, and
finally with SUNDIALS IDA, Section~\ref{sec:ida-comparison}. The benchmark
construction and implementation details are collected in the appendices.

%-------------------------------------
\section{Adapting GSAV to pH-DAEs}
\label{sec:gsav-to-phdae}
In this part, we first describe the system class, recall the GSAV scheme and tailor it to port-Hamiltonian (pH) systems.
\subsection{Port-Hamiltonian setting and energy geometry} \label{sec:Q-geometry}

We consider the following class of pH systems:
\begin{definition}[Our pH-DAE class]\label{def:pH-class}
The descriptor variable $x:[0,T] \to \R^{n}$ satisfies
\begin{subequations}\label{eq:nonlinear-pHDAE}
\begin{align} \label{eq:nonlinear-pHDAE-dynamics}
\frac{\mathrm d(Ex)}{\mathrm dt} 
&=
\bigl(J(x)-R(x)\bigr)e(x)+Bu(t),
\\
y
&=
B^\top e(x),
\end{align}
\end{subequations}
with given effort map $e(x)$ and
\begin{enumerate}[leftmargin=*,itemsep=1pt,topsep=1pt]
    \item \label{def-item:structure}\emph{Structure:} 
    \begin{itemize}
        \item[(a)] $J(x), R(x) \in \R^{n\times n}$ with
        $J(x)^\top=-J(x)$ and $R(x)=R(x)^\top\ge0$
        pointwise at every $x$.

        \item[(b)] $Q,E\in \R^{n\times n}$ with
        $Q=Q^\top\ge0$ and $E^\top Q=Q^\top E\ge0$.

        \item[(c)] $(E,\check A(x))$ is regular and of index one along a solution $x$, with $\check A(x)$ being the Jacobian of the
        right-hand side of \eqref{eq:nonlinear-pHDAE-dynamics}.
    \end{itemize}

    \item\label{def-item:input}
    \emph{Input regularity:}
    $u\in C^2([0,T],\R^m)$.

    \item
    Port matrix $B\in\R^{n\times m}$.

    \item
    \emph{Consistent initial value:}
    $x(0)=x_0\in\R^n$ such that there exists $v_0\in\R^n$ satisfying
    \[
    Ev_0
    =
    \bigl(J(x_0)-R(x_0)\bigr)e(x_0)+Bu(0).
    \]
    \item
    Hamiltonian $H$ satisfying
    $E^\top e(x)=\nabla H(x)$.

    \item
    \emph{Solution regularity:}
    $Ex\in C^3([0,T],\R^n)$ and
    $x\in C^2([0,T],\R^n)$.

    \item\label{def-item:H-bound}
    $H(x)$ is bounded from below: there exists $C>0$ such that
    \[
    H(x)>-C
    \qquad\text{for all }x\in\R^n.
    \]
    Choose $C_0>C$.
    \hfill $\Box$
    \end{enumerate}
\end{definition}

By the pH structure, smooth solutions satisfy the power balance
\[
\frac{\mathrm d}{\mathrm dt}H(x(t))
=
-e(x)^\top R(x)e(x)+u^\top y.
\]

Using matrix $Q$, we separate the linear part $Qx$ of the effort map such that the Hamiltonian splits accordingly, 
\begin{equation*} 
    e(x) = Qx + e_{\mathrm{nl}}(x), \qquad 
    H(x) = \tfrac{1}{2} x^\top E^\top Qx + H_{\mathrm{nl}}(x).
\end{equation*}

Next, we split the dynamics into a constant and a state-dependent
contribution. To this end, we decompose the structure and dissipation
matrices as
\begin{equation}\label{eq:JR-decomposition}
    J(x)=J_0+J_1(x),
    \qquad
    R(x)=R_0+R_1(x),
\end{equation}
where
\[
    J_0^\top=-J_0,
    \qquad
    J_1(x)^\top=-J_1(x),
    \qquad
    R_0=R_0^\top\ge0,
\]
and the total dissipation satisfies $R(x)\ge0$, with no separate
definiteness assumption on $R_1(x)$. We define the constant-core
operator
\[
    A:=-(J_0-R_0)Q.
\]

\begin{assumption}[Constant-core index-one structure]
\label{ass:split-nonlinearity}
For the splitting \eqref{eq:JR-decomposition}, the constant-core pencil
$(E,A)$ is regular and of index one.
\hfill $\Box$
\end{assumption}
Under Assumption~\ref{ass:split-nonlinearity}, substituting the
decompositions \eqref{eq:JR-decomposition} into
\eqref{eq:nonlinear-pHDAE-dynamics} yields the modified descriptor system

\begin{equation} \label{eq:modified-pHDAE}
    \frac{\diff(Ex)}{\diff t} +  A x + g(x) = Bu(t),
\end{equation}
where all state dependence is gathered into $g(x)$, leaving the linear operator $ A$ constant:
\begin{equation} \label{eq:A-g-definition}
    A := -(J_0 - R_0)Q, \qquad g(x) := -(J_0 - R_0)e_{\mathrm{nl}}(x) - \bigl(J_1(x) - R_1(x)\bigr)e(x).
\end{equation}

In contrast to the standard GSAV setting \cite{HuangShen2022}, the matrix $A$ is generally nonsymmetric ($A^\top = Q(J_0 + R_0) \neq A$ whenever $J_0 \neq 0$). Its structural positivity is revealed instead via the $Q$-weighted pairing
\begin{equation*} %
    \langle a,b\rangle_Q := a^\top Qb, \qquad \|a\|_Q^2 := a^\top Qa,
\end{equation*}
which is a semi-inner product with induced semi-norm $\|\cdot\|_Q$ and an inner product if $Q>0$. Since $QJ_0Q$ is skew-symmetric, the skew interconnection generates no energy, giving $Q$-accretivity:
\begin{equation} \label{eq:Q-dissipativity}
    \langle A x, x\rangle_Q = x^\top QJ_0Qx + x^\top QR_0Qx = (Qx)^\top R_0(Qx) \ge 0.
\end{equation}
This accretivity replaces the positive self-adjoint operator of standard GSAV.

Finally, solutions of \eqref{eq:modified-pHDAE} obey the power balance
\begin{equation} \label{eq:modified-energy-law}
    \frac{\diff}{\diff t}H(x) = -K(x) + \langle u,y\rangle, \qquad K(x) := e(x)^\top R(x)e(x) \ge 0, \quad y := B^\top e(x).
\end{equation}
Because energy may enter or exit through the external ports, $H(x)$ need not decrease monotonically; the auxiliary-variable dynamics are therefore driven by this open power balance rather than a closed-system decay law.
%-----------------------------------
\subsection{The GSAV scheme for pH-DAEs} \label{sec:GSAV-Scheme}
Using that $H$ is bounded from below, Def.~\ref{def:pH-class}.\ref{def-item:H-bound}, we can introduce the scalar auxiliary variable $r$
\begin{equation}\label{eq:r-and-h0}
r(t):=H(x(t))+C_0>0,
\qquad
h_0:=C_0-C>0.
\end{equation}
which expands \eqref{eq:modified-pHDAE} into the system that will be discretized in place of the original one,
    \begin{equation}\label{eq:expanded-conti-system}
        \frac{\diff (Ex)}{\diff t}+A x+g(x)=Bu(t), \qquad
        \frac{\diff r}{\diff t}=-K(x)+\langle u,y\rangle.
    \end{equation}
In this continuous setting, we have $r\equiv\H(x)$ and hence $\xi:=r/\H(x)\equiv1$. 

\paragraph{Time discretization: step-1.}
For time step $\delta t>0$ and step-averaged input $\bar{u}^{\ell + 1}:=\frac{1}{\delta t}\int_{t^\ell}^{t^{\ell + 1}}u(s)\,\diff s$, the relaxed IMEX BDF-$k$ scheme ($k\in\{1,2\}$) 
advances given $x^{\ell},\, r^\ell$ (and $\ell-1$ for $k=2$)
\begin{subequations}\label{eq:pHDAE-Discretization}
\begin{align}
\frac{\alpha_k E\bar{x}^{\ell + 1}-A_k[Ex^\ell]}{\delta t}+A\bar{x}^{\ell + 1}+g\bigl(G_k[x^\ell]\bigr)
&=BC_k[\bar{u}^{\ell + 1}], \label{eq:scheme-primal}\\
\frac{\tilde{r}^{\ell + 1}-r^\ell}{\delta t}
&=-\frac{\tilde{r}^{\ell + 1}}{H(\bar{x}^{\ell + 1})+C_0}K(\bar{x}^{\ell + 1})+\mathcal{P}_k^{\ell + 1}, \label{eq:scheme-scalar}\\
\xi^{\ell + 1}
&=\frac{\tilde{r}^{\ell + 1}}{H(\bar{x}^{\ell + 1})+C_0}, \label{eq:scheme-xi}\\
x^{\ell + 1}
&=\eta_k^{\ell + 1}\bar{x}^{\ell + 1}, \quad
\eta_k^{\ell + 1}:=1-\bigl(1-\xi^{\ell + 1}\bigr)^{k+1}, \label{eq:scheme-relax}
\end{align}
\end{subequations}
Here, $\xi^{\ell+1}$ and $\eta_k^{\ell+1}$ are auxiliary quantities
and $\tilde r^{\ell+1}$ is a preliminary value of the scalar variable,
whereas the initial data are $\bar x^0:=x^0$ and
$r^0:=\H(x^0)$. The final value $r^{\ell+1}$ is defined in
\eqref{eq:eop-passive} below.
The BDF operators and external port approximations ($y^j:=B^\top e(\bar x^j).$) are:
\begin{align*}
k=1\!: \quad & \alpha_1=1, \; A_1[Ex^\ell]=Ex^\ell, \; G_1[x^\ell]=x^\ell, \; C_1[\bar{u}^{\ell + 1}]=\bar{u}^{\ell + 1}, \; \mathcal{P}_1^{\ell + 1}:=\langle\bar u^{\ell + 1},y^{\ell + 1}\rangle,\\ %
k=2\!: \quad & \alpha_2=\tfrac{3}{2}, \; A_2[Ex^\ell]=2Ex^\ell-\tfrac{1}{2}Ex^{\ell-1}, \; G_2[x^\ell]=2x^\ell-x^{\ell-1},
\\ 
& C_2[\bar{u}^{\ell + 1}]:=\tfrac{3}{2}\bar{u}^{\ell + 1}-\tfrac{1}{2}\bar{u}^{\ell}, \qquad
\mathcal{P}_2^{\ell + 1}:=\left\langle\bar{u}^{\ell + 1},\tfrac{y^{\ell + 1}+y^\ell}{2}\right\rangle. \nonumber
\end{align*}
BDF-2 is initialized via BDF-1. With the discrete port work defined by
\begin{equation}\label{eq:discrete-port-work}
    \widehat W_k^{\ell+1} := \delta t\,\mathcal P_k^{\ell+1},
\end{equation}
the BDF-2 quadrature for exact output $y(t)=B^\top e(x(t))$ satisfies $\delta t\langle\bar u^{\ell+1}, \frac{y(t^{\ell+1})+y(t^\ell)}{2}\rangle = \int_{t_\ell}^{t^{\ell+1}}\langle u,y\rangle\,\diff t + \O(\delta t^3)$. While \eqref{eq:scheme-scalar} is only first-order, the relaxation factor $\eta_k^{\ell+1}$ restores order-$k$ state accuracy.
\paragraph{Time discretization: step-2 (adaptation to pH).}
For open port-Hamiltonian systems, the auxiliary energy need not decay monotonically due to energy exchange through the ports. We therefore replace the closed-system EOP bound $r^\ell$ in \cite[Thm.~2.5]{LiuZhangLi2024} by the discrete passivity budget
\[
s_k^{\ell + 1}:=r^\ell+\widehat W_k^{\ell + 1},
\]
and define the EOP projection
\begin{equation}\label{eq:eop-passive}
    r^{\ell + 1}
    :=
    \min\left\{
        H(x^{\ell + 1})+C_0,\;
        s_k^{\ell + 1}
    \right\}.
\end{equation}
Hence, we obtain the discrete passivity inequality:
\[
r^{\ell + 1}-r^\ell\le \widehat W_k^{\ell + 1}.
\]

For the numerical verification, let $x_{\mathrm{ref}}(t^\ell)$ denote the reference solution described in Section~\ref{sec:numerics-accuracy}, and define the global state and Hamiltonian errors by
\[
    \varepsilon_x:=\max_{0\le \ell \le N_t}\|x^\ell-x_{\mathrm{ref}}(t^\ell)\|_2,
    \qquad
    \varepsilon_H:=\max_{0\le \ell \le N_t}|H(x^\ell)-H(x_{\mathrm{ref}}(t^\ell))|.
\]
Figure~\ref{fig:main1} shows the observed 
temporal orders for the nonlinear benchmark of App.~\ref{app:benchmark}. BDF-1 exhibits
first-order convergence; BDF-2 and the implicit midpoint method
show second-order convergence in both $\varepsilon_x$ and $\varepsilon_H$. The order-$k$
convergence of the proposed schemes is established rigorously in
Theorem~\ref{thm:convergence}.

\begin{figure}[t]
  \centering
  \captionsetup[subfigure]{font=footnotesize,skip=1pt}

  \begin{subfigure}[b]{0.48\textwidth}
    \centering
    \includegraphics[
      width=\linewidth,
      trim={2 5 2 1},
      clip
    ]{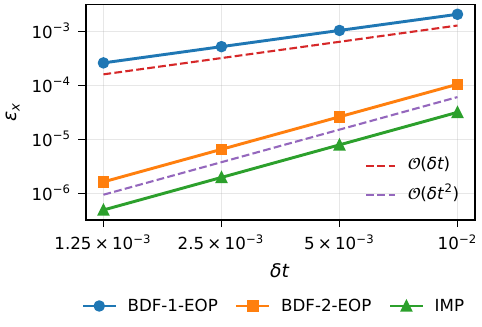}
    \caption{State error.}
    \label{fig:state-convergence}
  \end{subfigure}
  \hfill
  \begin{subfigure}[b]{0.48\textwidth}
    \centering
    \includegraphics[
      width=\linewidth,
      trim={2 5 2 1},
      clip
    ]{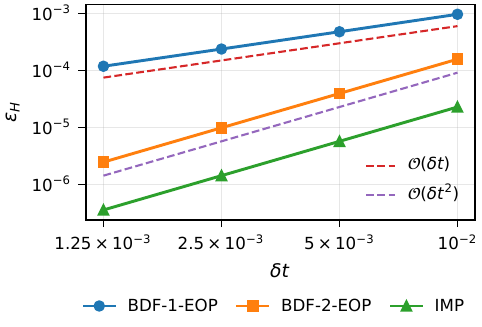}
    \caption{Hamiltonian error.}
    \label{fig:hamiltonian-convergence}
  \end{subfigure}

  \caption{Observed temporal convergence.}
  \label{fig:main1}
\end{figure}
Before turning to the convergence analysis, we summarize several structural and computational properties of the discretization.

\begin{remark}% 
\label{rem:properties}
\begin{enumerate}[leftmargin=*,itemsep=1pt,topsep=1pt]
    \item \emph{Linearity and Solvability:} Since all state-dependent terms are contained in $g$ and evaluated explicitly at $G_k[x^\ell]$, \eqref{eq:scheme-primal} is linear in $\bar x^{\ell + 1}$. By Proposition~\ref{prop:solvability}, $\frac{\alpha_k}{\delta t}E+A$ is invertible for all $\delta t>0$, and its factorization is reused across all steps for fixed $(k,\delta t)$.
    
    \item \emph{Explicit $\tilde{r}^{\ell + 1}$:} 
    \eqref{eq:scheme-scalar} can be stated in explicit form (with strictly positive denominator, since $K\ge0$, $H+C_0>0$):
    \[
        \tilde{r}^{\ell + 1}
        =\frac{r^\ell+\delta t\,\mathcal{P}_k^{\ell + 1}}
        {1+\delta t\,K(\bar x^{\ell + 1})/(H(\bar x^{\ell + 1})+C_0)}.
    \]

    \item \emph{Physical dissipation:} Evaluating $K(x)$ with the full state-dependent dissipation operator ensures that the auxiliary rate matches the physical energy law \eqref{eq:modified-energy-law}, which is essential for the consistency of $\xi^{\ell + 1}$.
    \item \emph{Extrapolation and Drift:} Evaluating $g$ at $G_k[x^\ell]$ follows the IMEX construction of \cite{HuangShen2022}. The resulting algebraic error residual enters the convergence analysis as $\rho_k^{\ell + 1}$ and is of order $\O(\delta t^k)$ under the assumptions of Section~\ref{sec:error-analysis}. \hfill $\Box$

\end{enumerate}
\end{remark}

%--------------------------------------------
\section{Analysis of the pH-DAE-GSAV scheme}
\label{sec:gsav-phdae-analysis}
%--------------------------------------------
	In this part, we analyze the projected GSAV scheme \eqref{eq:pHDAE-Discretization}--\eqref{eq:eop-passive} with $k\in\{1,2\}$, applied to \eqref{eq:expanded-conti-system}. We will establish
	convergence of order $k$ of the discrete solution $x^{\ell + 1}$ to $x(t^{\ell + 1})$, consistency of the relaxation scalar $\xi^{\ell + 1}$, and positivity and energy consistency of the projected auxiliary variable $r^{\ell + 1}$.

In the following, we write $(a^\top,b^\top)^\top$ for block column vectors.

\subsection{The index-one descriptor structure}\label{sec:wellposedness}

The analysis rests on splitting \eqref{eq:expanded-conti-system} into differential and algebraic parts via the index-one pencil $(E,A)$ (Ass.~\ref{ass:split-nonlinearity}). There exist nonsingular $S,T\in\R^{n\times n}$ such that
\begin{equation}\label{eq:kronecker}
SET=\begin{pmatrix}I_d&0\\0&0\end{pmatrix},\qquad
SAT=\begin{pmatrix}\tilde A_{11}&\tilde A_{12}\\\tilde A_{21}&\tilde A_{22}\end{pmatrix},
\end{equation}
with $\tilde A_{22}\in\R^{(n-d)\times(n-d)}$ nonsingular. In coordinates $T^{-1}x=(x_d^\top,x_a^\top)^\top$, let $P_d:=T\begin{pmatrix}I_d\\0\end{pmatrix}$ and $M_{dd}:=P_d^\top E^\top QP_d\in\R^{d\times d}$. Since $ET(0,x_a^\top)^\top=0$ and $E^\top Q=(E^\top Q)^\top$, we have $T^\top E^\top QT=\operatorname{blkdiag}(M_{dd},0)$, so the $E^\top Q$ semi-norm controls only $x_d$. Writing $Sg(T(x_d^\top,x_a^\top)^\top)=\bigl(g_d(x_d,x_a)^\top,g_a(x_d,x_a)^\top\bigr)^\top$, the algebraic component satisfies
\begin{equation}\label{eq:nonlinear-algebraic-constraint}
\tilde A_{21}x_d+\tilde A_{22}x_a+g_a(x_d,x_a)=(SBu)_a.
\end{equation}

\begin{assumption}\label{ass:DAE}
 For the pH model \eqref{eq:nonlinear-pHDAE} (Def.~\ref{def:pH-class}),
 we use the following DAE-related prerequisites: 
  \begin{enumerate}[leftmargin=*,itemsep=1pt,topsep=1pt]
      \item\label{ass-DAE-item:mass} In the index-one form \eqref{eq:kronecker} with differential inclusion $P_d$, the compression $M_{dd}:=P_d^\top E^\top QP_d>0$.

      \item\label{ass-DAE-item:regularity} \emph{Nonlinear regularity and constraint:} The maps $e_{\mathrm{nl}}$, $J_1$, and $R_1$ are $C^1$ on an open neighborhood of the relevant compact sets, so $g$ is locally Lipschitz. With $Sg(T(x_d^\top,x_a^\top)^\top)=\bigl(g_d(x_d,x_a)^\top,g_a(x_d,x_a)^\top\bigr)^\top$, the algebraic Jacobian $\tilde A_{22}+D_{x_a}g_a(x_d,x_a)$ is nonsingular on a neighborhood of the exact trajectory. \hfill $\Box$
  \end{enumerate}
\end{assumption}

\begin{remark}[Role of the regularity assumptions]
The regularity assumptions above are convenient sufficient conditions for
a unified BDF-1/BDF-2 analysis and are not intended to be minimal.
For order $k$, the consistency argument only requires the corresponding
order of temporal regularity of $Ex$, $x$, and $u$, while the $C^1$
assumptions on $e_{\mathrm{nl}}$, $J_1$, and $R_1$ are used to obtain
uniform local Lipschitz bounds for the nonlinear terms on the confinement
sets. Thus, the assumptions can be weakened accordingly, for example to
bounded weak derivatives of the required order and local Lipschitz
regularity. The nondegeneracy condition $M_{dd}>0$, in contrast, is
structural for the present energy argument, since it provides control of
the differential component through the $E^\top Q$ semi-norm. \hfill $\Box$
\end{remark}

By Assumption~\ref{ass:DAE}.\ref{ass-DAE-item:regularity},
$\tilde A_{22}+D_{x_a}g_a$ is nonsingular near the exact trajectory, so
\eqref{eq:nonlinear-algebraic-constraint} locally determines  
$x_a$ as a function of $(x_d,u)$.

Since $g$ is evaluated explicitly in the proposed scheme, the algebraic
component of the primal step is
\[
    \tilde A_{21}\bar x_d^{\ell + 1}
    +\tilde A_{22}\bar x_a^{\ell + 1}
    +g_a\bigl(G_k[x^\ell]\bigr)
    =
    \bigl(SBC_k[\bar u^{\ell + 1}]\bigr)_a.
\]
Thus only the constant block $\tilde A_{22}$ acts implicitly, and the
algebraic variable remains obtainable by a linear solve although the
continuous algebraic constraint may itself be nonlinear.

We now introduce the sets in which the discrete trajectories will be shown to remain. % to lie.
\begin{definition}[Confinement sets]\label{def:confinement-sets}
Let $\Gamma:=\{x(t):t\in[0,T]\}$ denote the exact trajectory and fix
$\rho_0>0$. The closed tubular neighborhood and associated BDF
extrapolation set are
\[
\mathcal T_{\rho_0}
:=
\{x\in\R^n:\operatorname{dist}(x,\Gamma)\le\rho_0\},
\qquad
\mathcal X_{\mathrm{ext}}
:=
\mathcal T_{\rho_0}\cup
\{2x-z:x,z\in\mathcal T_{\rho_0}\}.
\]
In index-one coordinates
$T^{-1}x=(x_d^\top,x_a^\top)^\top$, for $C_d,C_a>0$, define
\[
\mathcal X_{\mathrm{pre}}(C_d,C_a)
:=
\left\{
T(z_d^\top,z_a^\top)^\top:
\|z_d\|_2\le C_d,\;
\|z_a\|_2\le C_a
\right\},
\]
abbreviated as $\mathcal X_{\mathrm{pre}}$ when $C_d,C_a$ are the
uniform bounds from Lemma~\ref{lem:pre-relaxation}.
\end{definition}

\begin{proposition}[Solvability of the implicit step] \label{prop:solvability}
    Given the pH model~\eqref{eq:nonlinear-pHDAE} (Def.~\ref{def:pH-class}). Let $k\in\{1,2\}$ and let Ass.~\ref{ass:DAE}.\ref{ass-DAE-item:mass} hold. Then $\frac{\alpha_k}{\delta t}E+A$ is invertible for every $\delta t>0$. Consequently, the primal step \eqref{eq:scheme-primal} admits a unique solution $\bar x^{\ell + 1}$ 
    (for every $\ell$) %at every time level 
    with a matrix factorization that is constant across all time steps for fixed $k$ and $\delta t$.
\end{proposition}

\begin{proof}
    Let $v\in\mathbb{R}^n$ satisfy $(\frac{\alpha_k}{\delta t}E+A)v=0$. Testing with $v$ in the $Q$-weighted pairing and using \eqref{eq:Q-dissipativity} gives
    \[
        0 = \frac{\alpha_k}{\delta t}\langle Ev,v\rangle_Q + \langle A v,v\rangle_Q = \frac{\alpha_k}{\delta t}|v|_E^2 + (Qv)^\top R_0(Qv).
    \]
    Since $\alpha_k/\delta t>0$ and $R_0\ge0$, both terms are nonnegative, so $|v|_E^2 = 0$. In coordinates $T^{-1}v=(v_d^\top,v_a^\top)^\top$, we have $0 = |v|_E^2 = v_d^\top M_{dd}v_d$. Because $M_{dd}>0$ (Ass.~\ref{ass:DAE}.\ref{ass-DAE-item:mass}), $v_d=0$. 
    
    Transforming the homogeneous system via $S$ with $v_d=0$ reduces the algebraic block to $\tilde A_{22}v_a=0$. Since $\tilde A_{22}$ is nonsingular by the index-one structure, $v_a=0$, proving that $v=0$ and $\frac{\alpha_k}{\delta t}E+A$ is invertible. The factorization is independent of $\ell$ since $E$, $A$, $\alpha_k$, and $\delta t$ are constant.
\end{proof}
Since the nonlinear terms are evaluated at the intermediate state
$\bar x^{\ell + 1}$, the subsequent Lipschitz and boundedness estimates require
this state to remain in a fixed compact set. The following lemma provides
this preliminary confinement before the global bootstrap argument is available.

\begin{lemma}[Pre-relaxation confinement]\label{lem:pre-relaxation}
Given the pH model~\eqref{eq:nonlinear-pHDAE} (Def.~\ref{def:pH-class}) and Ass.~\ref{ass:DAE}, fix $\delta t_{\mathrm{pre}}>0$. Let $k\in\{1,2\}$, $0<\delta t\le\delta t_{\mathrm{pre}}$, and suppose $x^{\ell-j}\in\mathcal T_{\rho_0}$ for $j=0,\ldots,k-1$. Then there exist $C_d,C_a>0$, independent of $\ell$, $\delta t$, and $k$, such that $\|\bar z_d^{\ell+1}\|_2\le C_d$ and $\|\bar z_a^{\ell+1}\|_2\le C_a$. Consequently, $\bar x^{\ell+1}\in \mathcal X_{\mathrm{pre}}(C_d,C_a) =: \mathcal X_{\mathrm{pre}}$ by Definition~\ref{def:confinement-sets}, which is compact.
\end{lemma}

\begin{proof}
In coordinates \eqref{eq:kronecker}, set $z^j:=T^{-1}x^j$, $\bar z^{\ell+1}:=T^{-1}\bar x^{\ell+1}$, and $F_k^\ell := S\bigl(BC_k[\bar u^{\ell+1}]-g(G_k[x^\ell])\bigr) = \bigl((F_{d,k}^\ell)^\top,(F_{a,k}^\ell)^\top\bigr)^\top$. Eliminating $\bar z_a^{\ell+1} = \tilde A_{22}^{-1}(F_{a,k}^\ell-\tilde A_{21}\bar z_d^{\ell+1})$ from the transformed primal system yields
\begin{equation}\label{eq:pre-relaxation-schur}
    \bar z_d^{\ell+1} = \bigl(\alpha_kI_d+\delta t\,\tilde A_{\mathrm{Schur}}\bigr)^{-1}\bigl(A_k[z_d^\ell]+\delta t\,\widehat F_k^\ell\bigr),
\end{equation}
where $\tilde A_{\mathrm{Schur}} := \tilde A_{11}-\tilde A_{12}\tilde A_{22}^{-1}\tilde A_{21}$ and $\widehat F_k^\ell := F_{d,k}^\ell-\tilde A_{12}\tilde A_{22}^{-1}F_{a,k}^\ell$. By Proposition~\ref{prop:solvability} and nonsingularity of $\tilde A_{22}$, $\alpha_kI_d+\delta t\,\tilde A_{\mathrm{Schur}}$ is nonsingular for all $\delta t\ge0$, so $\max_{k\in\{1,2\}}\sup_{0\le\delta t\le\delta t_{\mathrm{pre}}}\|(\alpha_kI_d+\delta t\,\tilde A_{\mathrm{Schur}})^{-1}\| < \infty$. Compactness of $\mathcal T_{\rho_0}$ bounds $A_k[z_d^\ell]$ and ensures $G_k[x^\ell]\in\mathcal X_{\mathrm{ext}}$ (Definition~\ref{def:confinement-sets}); continuity of $g$ and boundedness of $u$ therefore bound $F_k^\ell$ and $\widehat F_k^\ell$ uniformly. Then \eqref{eq:pre-relaxation-schur} yields $\|\bar z_d^{\ell+1}\|_2\le C_d$, the algebraic relation gives $\|\bar z_a^{\ell+1}\|_2\le C_a$, and transforming via $T$ completes the proof.
\end{proof}

The analysis proceeds in three steps. First, the descriptor structure is used to recover control of the algebraic variables from the energy semi-norm. Second, a confinement argument provides uniform bounds for the explicitly evaluated nonlinear terms. These ingredients are then combined with BDF energy identities and the EOP relaxation to obtain convergence, positivity, and passivity.

\subsection{Relating the \texorpdfstring{$Q$}{Q}- and \texorpdfstring{$E$}{E}-semi-norms}

\begin{definition} \label{def:seminorms}
    Let $Q=Q^\top\ge0$ and $E^\top Q=Q^\top E\ge0$ as in Def.~\ref{def:pH-class}.
    We define the semi-norms on $\R^n$:
    \begin{equation} %\label{eq:Q-norm}
        \|x\|_Q := \sqrt{x^\top Qx}, \qquad |x|_E := \sqrt{x^\top E^\top Qx}.
        \tag*{$\Box$}
    \end{equation}
\end{definition}

\begin{lemma}[Control of the differential component] \label{lem:diff-control}
    Given the pH model~\eqref{eq:nonlinear-pHDAE} (Def.~\ref{def:pH-class}) and Ass.~\ref{ass:DAE}. 
    Let $T^{-1}x=(x_d^\top,x_a^\top)^\top$ for $x\in\R^n$ \eqref{eq:kronecker}. Then there exists $c:=\lambda_{\min}(M_{dd})^{-1/2}>0$ such that
    \begin{equation} \label{eq:diff-control}
        \|x_d\|_2 \le c\,|x|_E \qquad\text{for all }x\in\R^m.
    \end{equation}
\end{lemma}

\begin{proof}
    Since $T^\top E^\top QT = \operatorname{blkdiag}(M_{dd},0)$ by \eqref{eq:kronecker}, we have $|x|_E^2 = x_d^\top M_{dd}x_d \ge \lambda_{\min}(M_{dd})\|x_d\|_2^2$. Positivity $M_{dd}>0$ (Ass~\ref{ass:DAE}.\ref{ass-DAE-item:mass}) gives \eqref{eq:diff-control} with $c=\lambda_{\min}(M_{dd})^{-1/2}$.
\end{proof}

While $|\cdot|_E$ is blind to the algebraic component, the index-1 constraint recovers $x_a$ up to algebraic error residual:

\begin{proposition}[Semi-norm control of the error]\label{prop:error-control}
Given the pH model~\eqref{eq:nonlinear-pHDAE} (Def.~\ref{def:pH-class}) and Ass.~\ref{ass:DAE}, let $k\in\{1,2\}$. Suppose $\varepsilon^{\ell+1}:=x(t^{\ell+1})-x^{\ell+1}$ satisfies
\begin{equation}\label{eq:algebraic-error-relation}
    \tilde A_{21}\varepsilon_d^{\ell+1} + \tilde A_{22}\varepsilon_a^{\ell+1} = \rho_k^{\ell+1}, \qquad \|\rho_k^{\ell+1}\|_2 \le C_\rho\,\delta t^k.
\end{equation}
Then there exist $C_1,C_2>0$, independent of $\ell$, $\delta t$, and $k$, such that
\begin{equation}\label{eq:error-control}
    \|\varepsilon^{\ell+1}\|_2+\|\varepsilon^{\ell+1}\|_Q \le C_1|\varepsilon^{\ell+1}|_E+C_2\delta t^k.
\end{equation}
\end{proposition}

\begin{proof}
In coordinates $T^{-1}\varepsilon^{\ell+1} = ((\varepsilon_d^{\ell+1})^\top, (\varepsilon_a^{\ell+1})^\top)^\top$, Lemma~\ref{lem:diff-control} yields $\|\varepsilon_d^{\ell+1}\|_2 \le c|\varepsilon^{\ell+1}|_E$. Nonsingularity of $\tilde A_{22}$ and \eqref{eq:algebraic-error-relation} then give $\|\varepsilon_a^{\ell+1}\|_2 \le \|\tilde A_{22}^{-1}\|\bigl(\|\tilde A_{21}\|c|\varepsilon^{\ell+1}|_E + C_\rho\delta t^k\bigr)$. Since $\|\varepsilon^{\ell+1}\|_2 \le \|T\|(\|\varepsilon_d^{\ell+1}\|_2+\|\varepsilon_a^{\ell+1}\|_2)$ and $\|\varepsilon^{\ell+1}\|_Q \le \|Q^{1/2}\|\|\varepsilon^{\ell+1}\|_2$, collecting constants proves \eqref{eq:error-control}.
\end{proof}

Proposition~\ref{prop:error-control} closes the gap created by the degeneracy of $E^\top Q$: control of the differential component in $|\cdot|_E$, together with the algebraic error residual, yields control of the full state error. What remains is to ensure that all nonlinear evaluations stay in a region where the required bounds are uniform.

\subsection{Confinement without coercivity}
\label{sec:confinement}

Without coercivity of $H$, boundedness of the Hamiltonian does not
provide a bounded state domain \emph{a priori}. We therefore establish
confinement by a bootstrap argument around the exact trajectory using
the sets of Definition~\ref{def:confinement-sets}. Since
$x\in C^2([0,T])$, the trajectory $\Gamma$ is compact; hence
$\mathcal T_{\rho_0}$ and $\mathcal X_{\mathrm{ext}}$ are compact for
every fixed $\rho_0>0$.

By Ass.~\ref{ass:DAE}.\ref{ass-DAE-item:regularity}, $g$ is $C^1$ on an open
neighborhood of these sets and is therefore Lipschitz on
$\mathcal X_{\mathrm{ext}}$ with a constant $L_{\mathrm{ext}}$
independent of $\ell$, $\delta t$, and $k$. The first step is to show that,
whenever the BDF history states remain in $\mathcal T_{\rho_0}$, the
pre-relaxed state $\bar x^{\ell + 1}$ belongs to a fixed compact set
$\mathcal X_{\mathrm{pre}}$ of Definition~\ref{def:confinement-sets}.

\begin{lemma}[Inductive confinement]\label{lem:inductive-confinement}
    Given the pH model~\eqref{eq:nonlinear-pHDAE} (Def.~\ref{def:pH-class}) and Ass.~\ref{ass:DAE}.
    Let $k\in\{1,2\}$. Assume that the starting values required by the BDF-$k$ history satisfy
    \[
    x^j\in\mathcal{T}_{\rho_0},
    \qquad
    \|\varepsilon^j\|_2\le C(T)\,\delta t^k,
    \qquad j=0,\ldots,k-1,
    \]
    and that, whenever all previous states required by the BDF-$k$ history belong to $\mathcal{T}_{\rho_0}$, the preceding error estimates yield (for step $\ell$)
    \begin{equation}\label{eq:onestep}
        \|\varepsilon^{\ell+1}\|_2\le C(T)\,\delta t^k,
    \end{equation}
    where $C(T)$ is independent of $\ell$, $\delta t$, and $k$. Then, for
    \[
    \delta t\le\delta t_0:=\min\left\{1,\frac{\rho_0}{C(T)}\right\},
    \]
    one has
    \[
    x^\ell\in\mathcal{T}_{\rho_0}
    \qquad\text{and}\qquad
    \|\varepsilon^\ell\|_2\le C(T)\,\delta t^k
    \qquad\text{for all } \; \ell\delta t\le T.
    \]
\end{lemma}

\begin{proof}
    We proceed by induction. By assumption, starting values satisfy $x^j\in\mathcal{T}_{\rho_0}$ and $\|\varepsilon^j\|_2\le C(T)\delta t^k$ for $j=0,\ldots,k-1$. Assuming previous required BDF-$k$ states lie in $\mathcal{T}_{\rho_0}$, \eqref{eq:onestep} yields $\|\varepsilon^{\ell + 1}\|_2\le C(T)\delta t^k$. Since $x(t^{\ell + 1})\in\Gamma$,
    \[
        \operatorname{dist}(x^{\ell + 1},\Gamma) \le \|x^{\ell + 1}-x(t^{\ell + 1})\|_2 = \|\varepsilon^{\ell + 1}\|_2 \le C(T)\delta t^k \le C(T)\delta t_0 \le \rho_0
    \]
    for $k\in\{1,2\}$ and $\delta t\le\delta t_0\le1$. Thus $x^{\ell + 1}\in\mathcal{T}_{\rho_0}$, closing the induction.
\end{proof}

\subsection{Error analysis} \label{sec:error-analysis}

We now combine the previous ingredients into the global error estimate. The main difficulty is that the BDF energy identity (Lemma~\ref{lem:bdf-energy}, below) controls only the differential part of the error, while the explicit nonlinear term and the EOP relaxation couple differential, algebraic, and auxiliary variables. The proof therefore propagates these quantities simultaneously through a single induction.

\begin{lemma}[Port-power bound]\label{lem:port-power}
Given the pH model~\eqref{eq:nonlinear-pHDAE} (Def.~\ref{def:pH-class}) and Ass.~\ref{ass:DAE}, let $k\in\{1,2\}$ and suppose $\bar x^{\ell+1-j}\in\mathcal X_{\mathrm{pre}}$ for $j=0,\ldots,k-1$. Then there exists $P_*>0$, independent of $\ell$, $\delta t$, and $k$, such that
\begin{equation}\label{eq:port-power-bound}
    |\mathcal P_k^{\ell+1}|\le P_*.
\end{equation}
\end{lemma}

\begin{proof}
By input regularity (Def.~\ref{def:pH-class}.\ref{def-item:input}), $U_*:=\max_{t\in[0,T]}\|u(t)\|<\infty$, so $\|\bar u^{\ell+1}\|\le U_*$. Compactness of $\mathcal X_{\mathrm{pre}}$ and continuity of $e$ yield $Y_*:=\max_{z\in\mathcal X_{\mathrm{pre}}}\|B^\top e(z)\|<\infty$, ensuring $\|y^{\ell+1-j}\|\le Y_*$ across all active levels. Cauchy--Schwarz then gives $|\mathcal P_k^{\ell+1}| \le \|\bar u^{\ell+1}\|\max_{0\le j\le k-1}\|y^{\ell+1-j}\| \le U_*Y_* =: P_*$.
\end{proof}

\begin{remark}
Lemma~\ref{lem:port-power} requires only pre-relaxation confinement and therefore holds at level $\ell+1$ prior to establishing consistency of $\xi^{\ell+1}$, state error bounds, or positivity of $(r^{\ell+1},s_k^{\ell+1})$.
\end{remark}

	{Before stating the main convergence result, we provide two auxiliary lemmas.}
\begin{lemma}[BDF-$k$ consistency] \label{lem:bdf-consistency}
    Given the pH model~\eqref{eq:nonlinear-pHDAE} (Def.~\ref{def:pH-class}) and Ass.~\ref{ass:DAE}.
    We consider BDF-$k$ with its quantities given below \eqref{eq:pHDAE-Discretization} and $k\in\{1,2\}$. The exact solution of \eqref{eq:nonlinear-pHDAE} satisfies
    \begin{align}
        \frac{\alpha_k E x(t^{\ell + 1})-A_k[E x(t^\ell)]}{\delta t}+A x(t^{\ell + 1})+g\bigl(G_k[x(t^\ell)]\bigr)
        &=BC_k[\bar u^{\ell + 1}]+\tau_k^{\ell + 1}, \label{eq:continuous-bdfk}\\
        \|\tau_k^{\ell + 1}\|_2&\le C_\tau\delta t^k, \label{eq:bdfk-truncation}
    \end{align}
    where $C_\tau>0$ is independent of $\ell$, $\delta t$, and $k$.
\end{lemma}

\begin{proof}
    Taylor expansion about $t^{\ell + 1}$ gives
    \[
        \frac{\alpha_k E x(t^{\ell + 1})-A_k[E x(t^\ell)]}{\delta t}
        =\frac{\diff(Ex)}{\diff t}(t^{\ell + 1})+\O(\delta t^k),
        \qquad
        G_k[x(t^\ell)]=x(t^{\ell + 1})+\O(\delta t^k).
    \]
    Since $u\in C^2([0,T])$, $\bar u^{\ell + 1}=u(t^{\ell + 1})-\frac{\delta t}{2}u'(t^{\ell + 1})+\O(\delta t^2)$ and $\bar u^\ell=u(t^{\ell + 1})-\frac{3\delta t}{2}u'(t^{\ell + 1})+\O(\delta t^2)$, hence
    \[
        C_1[\bar u^{\ell + 1}]=u(t^{\ell + 1})+\O(\delta t),\qquad
        C_2[\bar u^{\ell + 1}]=\tfrac32\bar u^{\ell + 1}-\tfrac12\bar u^\ell
        =u(t^{\ell + 1})+\O(\delta t^2).
    \]
    Thus $C_k[\bar u^{\ell + 1}]=u(t^{\ell + 1})+\O(\delta t^k)$. Substitution into \eqref{eq:expanded-conti-system}, together with local Lipschitz continuity of $g$ on $\mathcal X_{\mathrm{ext}}$, proves \eqref{eq:continuous-bdfk}--\eqref{eq:bdfk-truncation}.
\end{proof}

\begin{lemma}[BDF-$k$ energy identity] \label{lem:bdf-energy}
    Let $b_E(v,w):=\langle Ev,w\rangle_Q=v^\top E^\top Qw$. Then
    \begin{align*}
        b_E(a-b,a)
        &=\tfrac12\bigl(|a|_E^2-|b|_E^2+|a-b|_E^2\bigr), \qquad k=1,\\ %
        b_E\left(\tfrac32a-2b+\tfrac12c,a\right)
        &=\tfrac14\bigl(|a|_E^2+|2a-b|_E^2\bigr)
          -\tfrac14\bigl(|b|_E^2+|2b-c|_E^2\bigr)
          +\tfrac14|a-2b+c|_E^2, \qquad k=2. %
    \end{align*}
\end{lemma}

\begin{proof}
    Both identities follow by direct algebraic expansion of the symmetric bilinear form $b_E$.
\end{proof}

The preceding consistency and $G$-stability identities provide the temporal ingredients; the descriptor estimates of Sections~\ref{sec:wellposedness}--~\ref{sec:confinement} supply the missing algebraic control.

\begin{theorem}[Convergence of the projected discrete solution] \label{thm:convergence}
    Given the pH model~\eqref{eq:nonlinear-pHDAE} (Def.~\ref{def:pH-class}) and Ass.~\ref{ass:DAE}. Let $k\in\{1,2\}$, and let $(\bar x^{\ell + 1},\tilde r^{\ell + 1},\xi^{\ell + 1},x^{\ell + 1},r^{\ell + 1})$ be computed by \eqref{eq:pHDAE-Discretization} and \eqref{eq:eop-passive} from consistent initial data $x^0=x(0)$ and $r^0=\H(x^0)$. For $k=2$, we assume that a proper initialization produced starting values satisfying
    \begin{equation} \label{eq:bdf2-start}
        \|\bar x^1-x(t^1)\|_2+\|x^1-x(t^1)\|_2+\|x^1-x(t^1)\|_Q\le C_{\mathrm{st}}\delta t^2,\qquad
        |1-\xi^1|\le C_{\mathrm{st}}\delta t,\qquad
        |r^1-\H(x^1)|\le C_{\mathrm{st}}\delta t^2 .
    \end{equation}
    Then there exists $\delta t_0>0$ such that for all $\delta t\le\delta t_0$
    and for all $\ell\ge1$ with $t_\ell\le T$ holds
    \begin{alignat*}{2}
        \sup_{\ell \,:\,t^\ell\le T}\|x(t^\ell)-x^\ell\|_2+\sup_{\ell \,:\,t^\ell\le T}\|x(t^\ell)-x^\ell\|_Q
        &\le C(T)\delta t^k, %
        & \qquad
        |1-\xi^\ell|&\le C_\xi\delta t, %
        \\
        |r^\ell-\H(x^\ell)|&\le C_Dt^\ell\delta t,
        \\ %
        r^\ell & \ge\tfrac{h_0}{2},
        &\qquad 
        s_k^{\ell + 1}&\ge\tfrac{h_0}{4} %
    \end{alignat*}
    (with $h_0$ from \eqref{eq:r-and-h0}).
    All constants are independent of $\ell$, $\delta t$, and $k$.
\end{theorem}

\begin{proof}
    Set $\varepsilon^\ell:=x(t^\ell)-x^\ell$, $\bar\varepsilon^{\ell + 1}:=x(t^{\ell + 1})-\bar x^{\ell + 1}$, $D^\ell:=r^\ell-\H(x^\ell)$, $\xi^0:=1$, and, writing $T^{-1}\varepsilon^\ell= ((\varepsilon_d^\ell)^\top,(\varepsilon_a^\ell)^\top)^\top$,
    \[
        \rho_k^\ell:=\tilde A_{21}\varepsilon_d^\ell+\tilde A_{22}\varepsilon_a^\ell.
    \]
    Define
    \begin{align*}
        \mathcal E_1^\ell&:=\tfrac12|\varepsilon^\ell|_E^2, &
        \mathcal E_2^\ell&:=\tfrac14\bigl(|\varepsilon^\ell|_E^2+|2\varepsilon^\ell-\varepsilon^{\ell-1}|_E^2\bigr),\\
        \bar{\mathcal E}_1^{\ell + 1}&:=\tfrac12|\bar\varepsilon^{\ell + 1}|_E^2, &
        \bar{\mathcal E}_2^{\ell + 1}&:=\tfrac14\bigl(|\bar\varepsilon^{\ell + 1}|_E^2+|2\bar\varepsilon^{\ell + 1}-\varepsilon^\ell|_E^2\bigr).
    \end{align*}
    We proceed by induction, carrying: (I)~$x^j\in\mathcal T_{\rho_0}$; (II)~$\mathcal E_k^j\le K_E(T)\delta t^{2k}$ and $\|\varepsilon^j\|_2+\|\varepsilon^j\|_Q\le C_E(T)\delta t^k$; (III)~$|1-\xi^j|\le C_\xi\delta t$; (IV)~$|D^j|\le C_Dt^j\delta t$; and (V)~$\|\rho_k^j\|_2\le C_\rho\delta t^k$. For $k=1$ these hold at $\ell=0$. For $k=2$, \eqref{eq:bdf2-start}, $\varepsilon^0=0$, and boundedness of the fixed matrices give $\mathcal E_2^1=\O(\delta t^4)$, $\|\rho_2^1\|_2=\O(\delta t^2)$, $|D^1|=\O(\delta t^2)=\O(t^1\delta t)$, and $x^1\in\mathcal T_{\rho_0}$ for sufficiently small $\delta t$.

    \paragraph{Pre-relaxation confinement.}
    By (I), Lemma~\ref{lem:pre-relaxation} gives $\bar x^{\ell + 1}\in\mathcal X_{\mathrm{pre}}$, so all required bounds and Lipschitz constants for $e$, $g$, $H$, and $K$ are uniform in $\ell$, $\delta t$, and $k$ on fixed compact sets.

    \paragraph{Intermediate error estimate.}
    Subtracting \eqref{eq:scheme-primal} from \eqref{eq:continuous-bdfk} yields
    \begin{equation} \label{eq:error-equation}
        \frac{\alpha_kE\bar\varepsilon^{\ell + 1}-A_k[E\varepsilon^\ell]}{\delta t}
        +A\bar\varepsilon^{\ell + 1}+\delta g_k^{\ell + 1}=\tau_k^{\ell + 1},
        \qquad
        \delta g_k^{\ell + 1}:=g(G_k[x(t^\ell)])-g(G_k[x^\ell]).
    \end{equation}
    Lipschitz continuity gives $\|\delta g_k^{\ell + 1}\|_2\le C_g\|G_k[\varepsilon^\ell]\|_2$. The algebraic block is
    \[
        \tilde A_{21}\bar\varepsilon_d^{\ell + 1}+\tilde A_{22}\bar\varepsilon_a^{\ell + 1}
        =:\bar\rho_k^{\ell + 1}=(S\tau_k^{\ell + 1})_a-(S\delta g_k^{\ell + 1})_a .
    \]
    Using $|\varepsilon^{\ell-1}|_E\le2|\varepsilon^\ell|_E+|2\varepsilon^\ell-\varepsilon^{\ell-1}|_E$ for $k=2$, Lemma~\ref{lem:diff-control}, nonsingularity of $\tilde A_{22}$, and (V) yield
    \[
        \|G_k[\varepsilon^\ell]\|_2^2\le C\mathcal E_k^\ell+C\delta t^{2k},
        \qquad
        \|\bar\rho_k^{\ell + 1}\|_2^2\le C\mathcal E_k^\ell+C\delta t^{2k},
    \]
    and hence
    \[
        \|\bar\varepsilon^{\ell + 1}\|_2+\|\bar\varepsilon^{\ell + 1}\|_Q
        \le C|\bar\varepsilon^{\ell + 1}|_E+C\|\bar\rho_k^{\ell + 1}\|_2.
    \]
    Testing \eqref{eq:error-equation} with $\bar\varepsilon^{\ell + 1}$ in $\langle\cdot,\cdot\rangle_Q$, multiplying by $\delta t$, and using Lemma~\ref{lem:bdf-energy} gives
    \begin{equation*} 
        \bar{\mathcal E}_k^{\ell + 1}-\mathcal E_k^\ell+\mathcal R_k^{\ell + 1}
        +\delta t\langle A\bar\varepsilon^{\ell + 1},\bar\varepsilon^{\ell + 1}\rangle_Q
        =\delta t\langle\tau_k^{\ell + 1}-\delta g_k^{\ell + 1},\bar\varepsilon^{\ell + 1}\rangle_Q,
    \end{equation*}
    where $\mathcal R_1^{\ell + 1}:=\frac12|\bar\varepsilon^{\ell + 1}-\varepsilon^\ell|_E^2\ge0$, $\mathcal R_2^{\ell + 1}:=\frac14|\bar\varepsilon^{\ell + 1}-2\varepsilon^\ell+\varepsilon^{\ell-1}|_E^2\ge0$, and $\langle A v,v\rangle_Q=v^\top QR_0Qv\ge0$. Moreover,
    \[
    \begin{aligned}
        \delta t\bigl|\langle\tau_k^{\ell + 1}-\delta g_k^{\ell + 1},\bar\varepsilon^{\ell + 1}\rangle_Q\bigr|
        &\le C\delta t\bigl((\mathcal E_k^\ell)^{1/2}+\delta t^k\bigr)
        \bigl(|\bar\varepsilon^{\ell + 1}|_E+\|\bar\rho_k^{\ell + 1}\|_2\bigr)\\
        &\le C\delta t\,\bar{\mathcal E}_k^{\ell + 1}
        +C\delta t\,\mathcal E_k^\ell+C\delta t^{2k+1},
    \end{aligned}
    \]
    by Cauchy--Schwarz and Young's inequality. Therefore, for sufficiently small $\delta t$,
    \begin{equation} \label{eq:intermediate-energy-recurrence}
        \bar{\mathcal E}_k^{\ell + 1}\le(1+C\delta t)\mathcal E_k^\ell+C\delta t^{2k+1}.
    \end{equation}
    By (II), $\bar{\mathcal E}_k^{\ell + 1}=\O(\delta t^{2k})$ and $\|\bar\varepsilon^{\ell + 1}\|_2+\|\bar\varepsilon^{\ell + 1}\|_Q=\O(\delta t^k)$.

    \paragraph{One-step energy consistency.}
    We have $\|\bar x^{\ell + 1}-x^\ell\|_2=\O(\delta t)$. Writing
    \begin{equation*} %
        \frac{\alpha_kE\bar x^{\ell + 1}-A_k[Ex^\ell]}{\delta t}
        =\frac{E(\bar x^{\ell + 1}-x^\ell)}{\delta t}+\chi_k^{\ell + 1},
        \qquad
        \chi_1^{\ell + 1}:=0,\quad
        \chi_2^{\ell + 1}:=\frac{E(\bar x^{\ell + 1}-2x^\ell+x^{\ell-1})}{2\delta t},
    \end{equation*}
    exact second differences and the error estimates give $\|\chi_k^{\ell + 1}\|_2=\O(\delta t)$. Moreover, for $\widehat x^{\ell + 1}:=G_k[x^\ell]$, $\|\widehat x^{\ell + 1}-\bar x^{\ell + 1}\|_2=\O(\delta t)$ and $\|C_k[\bar u^{\ell + 1}]-\bar u^{\ell + 1}\|_2=\O(\delta t)$. Hence \eqref{eq:scheme-primal} and \eqref{eq:A-g-definition} imply
    \[
        \frac{E(\bar x^{\ell + 1}-x^\ell)}{\delta t}
        =\bigl(J(\bar x^{\ell + 1})-R(\bar x^{\ell + 1})\bigr)e(\bar x^{\ell + 1})+B\bar u^{\ell + 1}+\O(\delta t).
    \]
    Since $\nabla H=E^\top e$, Taylor expansion, $J^\top=-J$, and $K(x)=e(x)^\top R(x)e(x)$ give
    \[
        \H(\bar x^{\ell + 1})-\H(x^\ell)
        =-\delta t\,K(\bar x^{\ell + 1})+\delta t\langle\bar u^{\ell + 1},y^{\ell + 1}\rangle+\O(\delta t^2).
    \]
    Furthermore $\mathcal P_k^{\ell + 1}=\langle\bar u^{\ell + 1},y^{\ell + 1}\rangle+\O(\delta t)$, with $y^{\ell + 1}-y^\ell=\O(\delta t)$ for $k=2$, and therefore
    \begin{equation} \label{eq:H-increment-unified}
        \H(\bar x^{\ell + 1})-\H(x^\ell)
        =-\delta t\,K(\bar x^{\ell + 1})+\delta t\,\mathcal P_k^{\ell + 1}+\O(\delta t^2).
    \end{equation}

    \paragraph{Consistency of the relaxation scalar.}
    Set $\widetilde D^{\ell + 1}:=\tilde r^{\ell + 1}-\H(\bar x^{\ell + 1})$. Subtracting \eqref{eq:H-increment-unified} from the scalar update yields
    \[
        \widetilde D^{\ell + 1}-D^\ell
        =-\delta t\,K(\bar x^{\ell + 1})(\xi^{\ell + 1}-1)+\O(\delta t^2).
    \]
    Since $\xi^{\ell + 1}-1=\widetilde D^{\ell + 1}/\H(\bar x^{\ell + 1})$,
    \begin{equation} \label{eq:D-recurrence-explicit}
        \widetilde D^{\ell + 1}
        =\frac{D^\ell+\O(\delta t^2)}{1+\delta t\,\kappa^{\ell + 1}},
        \qquad
        \kappa^{\ell + 1}:=\frac{K(\bar x^{\ell + 1})}{\H(\bar x^{\ell + 1})}\ge0.
    \end{equation}
    By (IV), $|\widetilde D^{\ell + 1}|\le C\delta t$; since $\H(\bar x^{\ell + 1})\ge h_0$,
    \begin{equation} \label{eq:xi-order}
        |1-\xi^{\ell + 1}|\le C_\xi\delta t.
    \end{equation}

    \paragraph{Relaxation and state error.}
    With $l^{\ell + 1}:=x^{\ell + 1}-\bar x^{\ell + 1}=-(1-\xi^{\ell + 1})^{k+1}\bar x^{\ell + 1}$, boundedness of $\bar x^{\ell + 1}$ and \eqref{eq:xi-order} give $\|l^{\ell + 1}\|_2=\O(\delta t^{k+1})$. Expanding $\mathcal E_k^{\ell + 1}$ through $\varepsilon^{\ell + 1}=\bar\varepsilon^{\ell + 1}-l^{\ell + 1}$ and applying Young's inequality gives
    \[
        \mathcal E_k^{\ell + 1}\le(1+C\delta t)\bar{\mathcal E}_k^{\ell + 1}+C\delta t^{2k+1}.
    \]
    Together with \eqref{eq:intermediate-energy-recurrence},
    \[
        \mathcal E_k^{\ell + 1}\le(1+C\delta t)\mathcal E_k^\ell+C\delta t^{2k+1},
    \]
    so discrete Gr\"onwall yields $\mathcal E_k^\ell\le K_E(T)\delta t^{2k}$. Moreover,
    \[
        \rho_k^{\ell + 1}
        =\bar\rho_k^{\ell + 1}-\tilde A_{21}l_d^{\ell + 1}-\tilde A_{22}l_a^{\ell + 1},
        \qquad
        \|\rho_k^{\ell + 1}\|_2\le C_\rho\delta t^k,
    \]
    closing (V). Lemma~\ref{lem:diff-control}, nonsingularity of $\tilde A_{22}$, and the last two bounds give
    \[
        \|\varepsilon^{\ell + 1}\|_2+\|\varepsilon^{\ell + 1}\|_Q\le C_E(T)\delta t^k,
    \]
    closing (II).

    \paragraph{Positivity and projection.}
    By (IV), $r^\ell\ge\H(x^\ell)-C_DT\delta t\ge h_0/2$ for sufficiently small $\delta t$. Lemma~\ref{lem:port-power} then gives
    \[
        s_k^{\ell + 1}=r^\ell+\delta t\,\mathcal P_k^{\ell + 1}
        \ge\tfrac{h_0}{2}-\delta tP_*\ge\tfrac{h_0}{4}>0
    \]
    for $\delta t\le h_0/(4P_*)$, and therefore $0<\tilde r^{\ell + 1}\le s_k^{\ell + 1}$. If $\H(x^{\ell + 1})\le s_k^{\ell + 1}$, then $D^{\ell + 1}=0$; otherwise $r^{\ell + 1}=s_k^{\ell + 1}<\H(x^{\ell + 1})$ and
    \[
        |D^{\ell + 1}|\le|\tilde r^{\ell + 1}-\H(x^{\ell + 1})|
        \le|D^\ell|+C_D\delta t^2,
    \]
    where \eqref{eq:D-recurrence-explicit} and $|\H(x^{\ell + 1})-\H(\bar x^{\ell + 1})|=\O(\delta t^{k+1})$ were used. Hence $|D^{\ell + 1}|\le C_Dt^{\ell + 1}\delta t$, closing (IV), and $r^{\ell + 1}\ge h_0/2$ for sufficiently small $\delta t$.

    \paragraph{Confinement.}
    Finally, $\operatorname{dist}(x^{\ell + 1},\Gamma)\le\|\varepsilon^{\ell + 1}\|_2\le C_E(T)\delta t^k\le\rho_0$, so Lemma~\ref{lem:inductive-confinement} closes (I). Taking $\delta t_0$ as the minimum of the preceding step-size restrictions completes the proof.
\end{proof}

\begin{remark}[Algebraic constraint accuracy]
    Let
    \[
        \mathcal C(x,t):=\bigl[S(A x+g(x)-Bu(t))\bigr]_a
    \]
    denote the algebraic residual of the continuous DAE. Since
    $\mathcal C(x(t),t)=0$ and $\mathcal C(\cdot,t)$ is locally Lipschitz,
    Theorem~\ref{thm:convergence} implies
    \[
        \|\mathcal C(x^\ell,t^\ell)\|_2\le C\delta t^k.
    \]
    Thus the projected state satisfies the nonlinear algebraic constraint
    with the same order as the state approximation. In the special case
    $g_a\equiv0$, additional exact constraint preservation may occur,
    depending on the algebraic forcing and the relaxation step. \hfill $\Box$
\end{remark}

\begin{proposition}[Energy fidelity of the EOP projection] \label{prop:eop-fidelity}
    Under the assumptions of Theorem~\ref{thm:convergence} with $\delta t\le\delta t_0$, the EOP projection $r^{\ell + 1} = \min\{\H(x^{\ell + 1}),s_k^{\ell + 1}\}$ is non-worsening with respect to the shifted %
    Hamiltonian:
    \begin{equation} \label{eq:nonworsening-cited}
        |r^{\ell + 1}-\H(x^{\ell + 1})| \le |\tilde r^{\ell + 1}-\H(x^{\ell + 1})|.
    \end{equation}
    Moreover, whenever the passivity budget is inactive, auxiliary energy tracks 
    $\H$
    exactly:
    \begin{equation} \label{eq:exact-fidelity}
        \H(x^{\ell + 1})\le s_k^{\ell + 1} \quad\Longrightarrow\quad r^{\ell + 1}=\H(x^{\ell + 1}).
    \end{equation}
\end{proposition}

\begin{proof}
    By Theorem~\ref{thm:convergence}, $s_k^{\ell + 1}>0$ and $\H(\bar x^{\ell + 1})\ge h_0>0$. The scalar update gives $0 < \tilde r^{\ell + 1} = \frac{s_k^{\ell + 1}}{1+\delta t\,K(\bar x^{\ell + 1})/\H(\bar x^{\ell + 1})} \le s_k^{\ell + 1}$. If $\H(x^{\ell + 1})\le s_k^{\ell + 1}$, the projection sets $r^{\ell + 1}=\H(x^{\ell + 1})$, proving \eqref{eq:exact-fidelity} and rendering \eqref{eq:nonworsening-cited} trivial. Otherwise, $0 < \tilde r^{\ell + 1} \le r^{\ell + 1} = s_k^{\ell + 1} < \H(x^{\ell + 1})$, which implies $|r^{\ell + 1}-\H(x^{\ell + 1})| = \H(x^{\ell + 1})-r^{\ell + 1} \le \H(x^{\ell + 1})-\tilde r^{\ell + 1} = |\tilde r^{\ell + 1}-\H(x^{\ell + 1})|$, confirming \eqref{eq:nonworsening-cited}.
\end{proof}

\begin{corollary}[Exact tracking of the Hamiltonian $H$] \label{cor:exact-energy-tracking}
    For the unshifted auxiliary energy $\widetilde H^\ell := r^\ell-C_0$,
    \begin{equation*} %
        H(x^{\ell + 1}) \le s_k^{\ell + 1}-C_0 \quad\Longrightarrow\quad \widetilde H^{\ell + 1} = H(x^{\ell + 1}).
    \end{equation*}
\end{corollary}

\begin{proof}
    Since $\H=H+C_0$, the condition $H(x^{\ell + 1})\le s_k^{\ell + 1}-C_0$ is equivalent to $\H(x^{\ell + 1})\le s_k^{\ell + 1}$. Proposition~\ref{prop:eop-fidelity} then yields $r^{\ell + 1} = \H(x^{\ell + 1}) = H(x^{\ell + 1})+C_0$, so $\widetilde H^{\ell + 1} = r^{\ell + 1}-C_0 = H(x^{\ell + 1})$.
\end{proof}

\begin{remark}[Role of budget positivity] \label{rem:nonworsening-general}
    Budget positivity is essential: $s_k^{\ell + 1}>0$ ensures $0 < \tilde r^{\ell + 1}\le s_k^{\ell + 1}$, placing the projected value $r^{\ell + 1}$ between $\tilde r^{\ell + 1}$ and $\H(x^{\ell + 1})$ when the budget is active ($\tilde r^{\ell + 1}\le r^{\ell + 1}=s_k^{\ell + 1}<\H(x^{\ell + 1})$). If $s_k^{\ell + 1}<0$, this ordering need not hold, so the
    non-worsening property is no longer guaranteed. Thus, the positivity
    established in Theorem~\ref{thm:convergence} ensures that the EOP
    correction is energy-non-worsening.
\end{remark}

\begin{corollary}[Boundedness of the auxiliary variables] \label{cor:auxiliary-boundedness}
    Under the assumptions of Theorem~\ref{thm:convergence}, there exists $R_*>0$, independent of $\ell$, $\delta t$, and $k$, such that for all $t^{\ell + 1}\le T$:
    \begin{equation} \label{eq:auxiliary-boundedness}
        \tfrac{h_0}{2} \le r^\ell \le R_*, \qquad 0 < \tilde r^{\ell + 1} \le R_*.
    \end{equation}
\end{corollary}

\begin{proof}
    The lower bound on $r^\ell$ comes from Theorem~\ref{thm:convergence}. By definition of the projection, $r^\ell\le\H(x^\ell)$. Since $x^\ell\in\mathcal T_{\rho_0}$ is compact, $r^\ell \le H_* := \max_{z\in\mathcal T_{\rho_0}}\H(z) < \infty$. Using Lemma~\ref{lem:port-power} ($|\mathcal P_k^{\ell + 1}|\le P_*$) and Proposition~\ref{prop:eop-fidelity} ($\tilde r^{\ell + 1}\le s_k^{\ell + 1}$), we have $0 < \tilde r^{\ell + 1} \le s_k^{\ell + 1} = r^\ell+\delta t\,\mathcal P_k^{\ell + 1} \le H_*+\delta t_0P_* =: R_*$, proving \eqref{eq:auxiliary-boundedness}.
\end{proof}

The preceding results describe the relation between the auxiliary and
physical energies. Figure~\ref{fig:energy-fidelity} illustrates this
behavior for the nonlinear benchmark~\ref{app:benchmark}. The EOP budget
remains inactive throughout this experiment, so the auxiliary energy
tracks $H(x^\ell)+C_0$ to roundoff accuracy, in agreement with
Corollary~\ref{cor:exact-energy-tracking}.

\begin{figure}[t]
  \centering
  \captionsetup[subfigure]{font=footnotesize,skip=1pt}

  \begin{subfigure}[b]{0.465\textwidth}
    \centering
    \includegraphics[width=\linewidth]{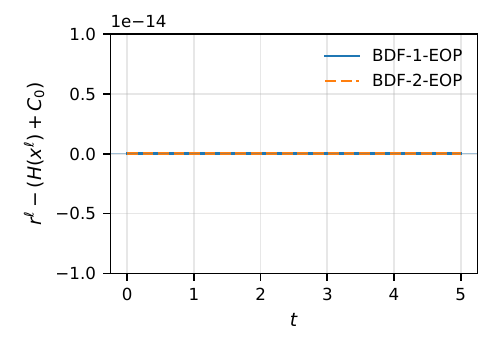}
    \caption{Auxiliary-energy fidelity.}
    \label{fig:auxiliary-energy-fidelity}
  \end{subfigure}
  \hfill
  \begin{subfigure}[b]{0.465\textwidth}
    \centering
    \includegraphics[width=\linewidth]{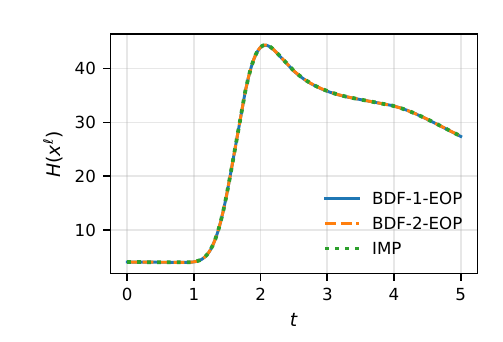}
    \caption{Physical Hamiltonian evolution.}
    \label{fig:hamiltonian-evolution}
  \end{subfigure}

  \caption{Auxiliary and physical energy behavior for the nonlinear benchmark.}
  \label{fig:energy-fidelity}
\end{figure}

\subsection{Discrete passivity} \label{sec:discrete-passivity}

The continuous port-Hamiltonian system satisfies $\frac{\diff}{\diff t}H(x) = -K(x)+\langle u,y\rangle \le \langle u,y\rangle$, making $H$ a storage function. Discretely, the auxiliary energy satisfies a discrete dissipation inequality with respect to $\mathcal P_k^{\ell + 1}$; when the EOP budget is inactive, Proposition~\ref{prop:eop-fidelity} identifies this auxiliary storage with the physical Hamiltonian.

\begin{lemma}[Discrete auxiliary passivity] \label{lem:auxiliary-passivity}
    Let $k\in\{1,2\}$ and let $r^{\ell + 1}$ be computed by \eqref{eq:eop-passive}. Then
    \begin{equation} \label{eq:auxiliary-passivity}
        r^{\ell + 1}-r^\ell \le \delta t\,\mathcal P_k^{\ell + 1}, \qquad\text{or equivalently,}\qquad \widetilde H^{\ell + 1}-\widetilde H^\ell \le \delta t\,\mathcal P_k^{\ell + 1},
    \end{equation}
    where $\widetilde H^\ell:=r^\ell-C_0$. In particular, $r^{\ell + 1}\le r^\ell$ and $\widetilde H^{\ell + 1}\le\widetilde H^\ell$ for unforced systems ($u\equiv0$, $\mathcal P_k^{\ell + 1}=0$).
\end{lemma}

\begin{proof}
    By the definition of $s_k^{\ell + 1}$ and \eqref{eq:eop-passive}, $r^{\ell + 1} = \min\{\H(x^{\ell + 1}),s_k^{\ell + 1}\} \le s_k^{\ell + 1} = r^\ell+\delta t\,\mathcal P_k^{\ell + 1}$, yielding the first bound in \eqref{eq:auxiliary-passivity}. Since $\widetilde H^\ell=r^\ell-C_0$ with constant $C_0$, $\widetilde H^{\ell + 1}-\widetilde H^\ell = r^{\ell + 1}-r^\ell$, proving equivalence.
\end{proof}

\begin{remark}[Auxiliary versus physical passivity] \label{rem:auxiliary-physical-passivity}
    Inequality \eqref{eq:auxiliary-passivity} is an algebraic consequence of the projection and holds unconditionally without requiring $\tilde r^{\ell + 1}>0$ or $s_k^{\ell + 1}>0$. Budget positivity is needed in Theorem~\ref{thm:convergence} and Proposition~\ref{prop:eop-fidelity} only to establish $\tilde r^{\ell + 1}>0$ and energy non-worsening. Furthermore, \eqref{eq:auxiliary-passivity} governs the auxiliary storage $\widetilde H^\ell$; it implies the physical passivity inequality $H(x^{\ell + 1})-H(x^\ell) \le \delta t\,\mathcal P_k^{\ell + 1}$ on steps where the EOP budget is inactive at both consecutive levels (via Corollary~\ref{cor:exact-energy-tracking}).
\end{remark}

%------------------------------
\section{Numerical results}
\label{sec:numerics-gsav-phdae}

Next, we test the numerical scheme
\eqref{eq:pHDAE-Discretization} with the correction
\eqref{eq:eop-passive}, henceforth denoted BDF-$k$-EOP, $k \in \{1 , 2 \}$, against the
implicit midpoint method (IMP). For IMP, the nonlinear system at each
time step is solved using variants of Newton's method, which we compare.

\subsection{Computational environment}
\label{sec:numerical-setup}

The numerical experiments are conducted primarily on the PLEIADES computing cluster (at the University of Wuppertal). The main numerical campaign uses a dual-socket
compute node equipped with two AMD EPYC 7452 (Rome) processors, with
$64$ physical CPU cores in total. 
This includes convergence, structure-preservation,
Newton-solver, work--precision, nonlinearity, and problem-size scaling
studies for the large-scale benchmark (Section~\ref{sec:large_scale_benchmark}).

All PLEIADES timings are single-core measurements. The process is pinned
to one physical CPU core, simultaneous multithreading is disabled, and
OpenMP and BLAS thread counts are restricted to one. This removes
thread-level parallelism from the comparison and isolates the
single-core computational cost of the time integrators.

The proposed schemes are implemented in Python using NumPy 2.3.2 and
SciPy 1.16.1, without JIT compilation or problem-specific compiled
kernels. Sparse matrix operations and sparse direct factorizations are
provided by SciPy. No method is supplied with method-specific low-level
optimizations, so the comparison reflects the different algebraic and
solver work induced by the time discretizations rather than specialized
kernel implementations.

The low-dimensional nonlinear stress test below is performed on a laptop
equipped with an Intel Core Ultra 7 268V processor with eight physical
cores ($4$ performance and $4$ efficiency cores). The comparison with
SUNDIALS \cite{hindmarsh2005sundials} IDA is also performed on this
machine, since the \texttt{sundials4py} interface is not available on
PLEIADES. For the IDA comparison, all BDF-2-EOP timings are therefore
obtained on the same laptop and in the same software environment. IDA is
accessed through \texttt{sundials4py}~7.8.0 and uses the native SUNDIALS
banded direct linear solver.

All reported wall-clock results are obtained from repeated complete time
integrations, with the median runtime used for comparison. Matrix
assembly, factorizations, and time stepping are included in the reported
solver costs. Solver tolerances, time-step sizes, and numbers of
repetitions are specified for the individual experiments.

The code for the experiments conducted below can be found at 
\begin{equation*}
    \text{https://github.com/cacheFriendlySharma/pH-DAE-EOP-GSAV/tree/main}
\end{equation*}

\subsection{Low-dimensional nonlinear stress testing}
\label{sec:toy-stress}

We first isolate the effect of strongly state-dependent nonlinear operators
in a small index-one problem for which Jacobian reuse becomes difficult.
Let $x=(q,p,z)^\top$ and
\[
E=\text{diag}(1,1,0),\qquad B=(0,1,0)^\top,\qquad
H(x)=\frac12(q^2+p^2)
+\frac{\mu}{\beta}
\left[
\log\!\cosh(\beta q)-\frac12\tanh^2(\beta q)
\right],
\]
with
\[
e(x)=
\begin{bmatrix}
q+\mu\tanh^3(\beta q)&p&z
\end{bmatrix}^{\!\top},
\qquad E^\top e(x)=\nabla H(x).
\]
The interconnection and dissipation operators are
\[
J(x)=
\begin{bmatrix}
0 & \omega+\psi(q) & 0\\
-\omega-\psi(q) & 0 & -c\\
0 & c & 0
\end{bmatrix},
\qquad
R(x)=\text{diag}\!\bigl(\nu+\chi(q),\,\nu+\chi(q),\,\rho\bigr),
\]
where
\[
\psi(q)=\kappa\tanh^3(\beta q),
\qquad
\chi(q)=\gamma\tanh^2(\beta q).
\]
The algebraic equation is $cp-\rho z=0$. Thus the Hamiltonian,
interconnection, and dissipation are nonlinear but bounded, whereas their
derivatives can vary sharply with the state. Moreover, the nonlinear
remainder has zero algebraic component, so BDF-$k$-EOP retains the
constant implicit matrix $\alpha_kE/\delta t+A$.

We set
$\omega=c=\rho=1$, $\nu=\gamma=0.05$, $\mu=0.5$, $\kappa=1$,
and $\beta=32$, with $A_u=2$, $t_0=0.5$, $\tau=2$, and $T=5$.
These parameters were fixed from the continuous reference trajectory
before evaluating any Newton solver. With
\[
f(x):=(J(x)-R(x))e(x),
\qquad
\max_t\|Df(x(t))-Df(0)\|_2\approx27,
\]
the test produces substantial Jacobian variation without unbounded
nonlinear forcing.

At the coarsest tested step, $\delta t=0.1$, full Newton converges for all
steps, whereas modified and simple Newton (with frozen Jacobian) reach the prescribed limit of
$20$ corrections at $t=2.6$. 	BDF-2-EOP completes the integration for all steps, see Table~\ref{tab:toy-robustness}.

\begin{table}[htb]
    \centering
    \footnotesize
    \setlength{\tabcolsep}{4pt}
    \renewcommand{\arraystretch}{0.95}
    \caption{Solver robustness for the nonlinear stress test at
    $\delta t=0.1$; the Newton tolerance is $10^{-10}$ with at most
    $20$ corrections per step.}
    \label{tab:toy-robustness}
    \begin{tabular}{@{}lrrrrrr@{}}
        \toprule
        Method
        & $\varepsilon_x$
        & Time [s]
        & Mean Newton
        & Max Newton
        & Fact.
        & Solves \\
        \midrule
        BDF-2-EOP
        & $3.258{\times}10^{-2}$
        & $0.00255$
        & -- & --
        & $2$
        & $50$ \\
        IMP-full
        & $8.220{\times}10^{-3}$
        & $0.00570$
        & $2.40$
        & $5$
        & $120$
        & $120$ \\
        IMP-modified
        & fail
        & --
        & $3.31$
        & $20$
        & $26$
        & $86$ \\
        IMP-simple
        & fail
        & --
        & $6.08$
        & $20$
        & $1$
        & $158$ \\
        \bottomrule
    \end{tabular}
\end{table}

To quantify this effect, let
\[
M_\ell:=\frac{E}{\delta t}-\frac12Df(x_{\ell+1/2}),
\qquad
M_0:=\frac{E}{\delta t}-\frac12Df(0)
\]
denote the midpoint Jacobian and its globally frozen counterpart.
For $\delta t=0.1$,
\[
\max_\ell \|I-M_0^{-1}M_\ell\|_2=1.48,
\]
so the standard $2$-norm contraction condition for the frozen iteration
is not available. This does not by itself imply divergence. Under
refinement the same quantity decreases to $0.11$ at
$\delta t=6.25\times10^{-3}$, although
$\max_\ell \|Df(x_{\ell +1/2})-Df(0)\|_2$ remains approximately $23$--$27$.
Thus simple Newton recovers as the $E/\delta t$ contribution increasingly
dominates the midpoint Jacobian.

A near-matched state-accuracy comparison is given in
Table~\ref{tab:toy-matched}. BDF-2-EOP requires two factorizations,
compared with $120$ for full Newton.

\begin{table}[htb]
    \centering
    \footnotesize
    \setlength{\tabcolsep}{4pt}
    \renewcommand{\arraystretch}{0.95}
    \caption{Near-matched state accuracy for the nonlinear stress test;
    timings are median laptop runtimes.}
    \label{tab:toy-matched}
    \begin{tabular}{@{}lrrrrr@{}}
        \toprule
        Method
        & $\delta t$
        & $\varepsilon_x$
        & $\varepsilon_H$
        & Time [s]
        & Fact. \\
        \midrule
        BDF-2-EOP
        & $0.05$
        & $8.982{\times}10^{-3}$
        & $7.567{\times}10^{-3}$
        & $0.00464$
        & $2$ \\
        IMP-full
        & $0.10$
        & $8.220{\times}10^{-3}$
        & $4.667{\times}10^{-3}$
        & $0.00570$
        & $120$ \\
        \bottomrule
    \end{tabular}
\end{table}

This example is not intended to establish a universal work--precision
advantage, but to isolate the robustness benefit of a fixed implicit
linear solve when the nonlinear Jacobian varies strongly.

\subsection{Large-scale nonlinear benchmark}\label{sec:large_scale_benchmark}
The preceding example is 
chosen to expose the distinction
	between the fixed BDF-$k$-EOP solve and the fully nonlinear implicit solves required by IMP.
It is comparatively easy to demonstrate this advantage when Jacobian
reuse is difficult. We therefore next consider the large-scale
FPU-$\beta$/Maxwell benchmark Appendix~\ref{app:benchmark}, which is
empirically much more favorable to Newton-type methods. In this problem,
the state dependence is introduced through local spring and damping
terms superimposed on a dominant constant chain and Maxwell structure;
in particular, a globally frozen midpoint Jacobian remains effective
over the complete trajectory. This provides a %deliberately 
more
demanding comparison for BDF-$k$-EOP, since the inexpensive Newton
variants operate close to their most favorable regime.

Unless stated otherwise, all remaining experiments use this nonlinear
FPU-$\beta$/Maxwell benchmark with the fixed parameters listed in
Table~\ref{tab:benchmark-parameters}. The system size $N$, time step
$\delta t$, and nonlinear- or adaptive-solver tolerances are varied only
in the experiments in which they constitute the quantity under study.

\begin{table}[htb]
    \centering
    \footnotesize
    \caption{Fixed parameters of the nonlinear FPU-$\beta$/Maxwell benchmark.}
    \label{tab:benchmark-parameters}
    \begin{tabular}{lll}
        \toprule
        Quantity & Symbol & Value \\
        \midrule
        Heavy mass                     & $m_{\mathrm h}$ & $1$ \\
        Light mass                     & $m_{\mathrm l}$ & $0.5$ \\
        FPU nonlinearity               & $\beta$         & $4$ \\
        Strain-coordinate nonlinearity & $\beta_J$       & $4$ \\
        Nonlinear damping              & $\beta_R$       & $2$ \\
        Maxwell stiffness              & $K_M$           & $2$ \\
        Maxwell viscosity              & $\eta_i$        & $0.4$ \\
        Initial strain amplitude       & $a_0$           & $0.5$ \\
        Auxiliary-energy shift         & $C_0$           & $1$ \\
        Initial-profile wavelength     & $L_0$        & $16$ \\
        Forcing wavelength             & $L_f$        & $16$ \\
        Input amplitude                & $A_u$           & $2$ \\
        Input start time               & $t_0$           & $0.5$ \\
        Input duration                 & $\tau$           & $2$ \\
        Final time                     & $T$              & $5$ \\
        \bottomrule
    \end{tabular}
\end{table}

\subsubsection{Accuracy and structure verification}
\label{sec:numerics-accuracy}
\paragraph{Temporal accuracy.}
For the above presented temporal order study, Fig.~\ref{fig:main1}, the reference solution is computed independently by eliminating the
algebraic variable and integrating the resulting reduced ODE with the
adaptive eighth-order DOP853 (Dormand-Prince) method using
$\mathrm{rtol}=10^{-12}$ and $\mathrm{atol}=10^{-14}$.
The algebraic variable is then reconstructed from the constraint.
Reference accuracy is verified by recomputing the solution with
$\mathrm{rtol}=2\times10^{-13}$ and $\mathrm{atol}=2\times10^{-15}$.
For each error metric, the discrepancy between the two reference solutions is kept below $2\%$ of the minimum finest-step error of BDF-2-EOP and implicit midpoint (IMP).

\paragraph{Structure verification.}
Table~\ref{tab:structure-diagnostics} compares discrete passivity and physical energy balance across BDF-$k$-EOP and IMP. For step $t^\ell\to t^{\ell+1}$ with method-specific discrete port power $P_\ell$ and dissipation $K_\ell$, we define the local and accumulated energy-balance defects
\[
    B_\ell := H(x^{\ell+1})-H(x^\ell) - \delta t\,P_\ell + \delta t\,K_\ell, \qquad D_H^j := \sum_{\ell=0}^{j-1}B_\ell.
\]
We monitor $\max_\ell|B_\ell|$ and $\max_j|D_H^j|$, alongside the auxiliary and physical passivity violations ($[a]_+ := \max\{a,0\}$)
\[
    V_{\mathrm{aux}} := \max_\ell \bigl[r^{\ell+1}-r^\ell-\delta t\,P_\ell\bigr]_+, \qquad V_H := \max_j \left[ H(x^j)-H(x^0) - \delta t\sum_{\ell=0}^{j-1}P_\ell \right]_+,
\]
where $V_{\mathrm{aux}}$ applies solely to the EOP formulations.

\begin{table}[htb]
    \centering
    \footnotesize
    \caption{Structure diagnostics for the dissipative benchmark.}
    \label{tab:structure-diagnostics}
    \begin{tabular}{@{}lrrrrr@{}}
        \toprule
        Method
        & $\delta t$
        & $V_{\mathrm{aux}}$
        & $V_H$
        & $\max_\ell|B_\ell|$
        & $\max_j|D_H^j|$ \\
        \midrule
        BDF-1-EOP
        & $1.0{\times}10^{-2}$ & $0$ & $0$
        & $6.958{\times}10^{-3}$ & $5.265{\times}10^{-1}$ \\
        & $5.0{\times}10^{-3}$ & $0$ & $0$
        & $1.740{\times}10^{-3}$ & $2.638{\times}10^{-1}$ \\
        & $2.5{\times}10^{-3}$ & $0$ & $0$
        & $4.351{\times}10^{-4}$ & $1.321{\times}10^{-1}$ \\
        \addlinespace
        BDF-2-EOP
        & $1.0{\times}10^{-2}$ & $0$ & $0$
        & $1.231{\times}10^{-3}$ & $7.867{\times}10^{-2}$ \\
        & $5.0{\times}10^{-3}$ & $0$ & $0$
        & $3.162{\times}10^{-4}$ & $4.064{\times}10^{-2}$ \\
        & $2.5{\times}10^{-3}$ & $0$ & $0$
        & $8.014{\times}10^{-5}$ & $2.064{\times}10^{-2}$ \\
        \addlinespace
        IMP
        & $1.0{\times}10^{-2}$ & -- & $0$
        & $6.658{\times}10^{-7}$ & $5.683{\times}10^{-5}$ \\
        & $5.0{\times}10^{-3}$ & -- & $0$
        & $8.323{\times}10^{-8}$ & $1.421{\times}10^{-5}$ \\
        & $2.5{\times}10^{-3}$ & -- & $0$
        & $1.040{\times}10^{-8}$ & $3.552{\times}10^{-6}$ \\
        \bottomrule
    \end{tabular}
\end{table}

As predicted by Lemma~\ref{lem:auxiliary-passivity},
BDF-1-EOP and BDF-2-EOP exhibit zero auxiliary passivity violation at
all tested step sizes. Physical passivity is also observed empirically
for all three methods, although only the auxiliary inequality is
guaranteed by the EOP construction. Under refinement, the EOP balance
defects exhibit approximately second-order local and first-order
cumulative decay, whereas IMP exhibits approximately third-order local
and second-order cumulative decay.

\subsubsection{Time-integrator comparison}
\label{sec:time-comparison}

We first compare the linear and nonlinear work of the second-order
methods. At $N=2048$ and $\delta t=5\times10^{-3}$,
BDF-2-EOP requires one linear solve per step after the two
factorizations associated with the BDF-1-EOP start and BDF-2-EOP update.
The midpoint variants require substantially more linear work because
of their nonlinear iterations; see Table~\ref{tab:solver-work}.

\begin{table}[htb]
\centering
\footnotesize
\setlength{\tabcolsep}{5pt}
\renewcommand{\arraystretch}{0.95}
\caption{Solver work at $N=2048$ and
$\delta t=5\times10^{-3}$ over $1000$ time steps.}
\label{tab:solver-work}
\begin{tabular}{@{}lrrrrr@{}}
\toprule
Method & Mean iter. & Residuals & Jacobians & LU & Solves \\
\midrule
BDF-2-EOP & --   & --   & --   & 2    & 1000 \\
IMP-full      & 2.00 & 3000 & 2000 & 2000 & 2000 \\
IMP-modified  & 2.91 & 3913 & 1000 & 1000 & 2913 \\
IMP-simple    & 5.95 & 6953 & 1    & 1    & 5953 \\
\bottomrule
\end{tabular}
\end{table}

\paragraph{Matched accuracy.}
Fixed-step timings alone do not account for differences in accuracy.
Tables~\ref{tab:wp-state} and~\ref{tab:wp-energy} therefore compare
runtimes at matched global state and Hamiltonian errors. Runtimes are
obtained by log--log interpolation between measured work--precision
points; parenthetical values denote the ratio relative to
BDF-2-EOP.

\begin{table}[htb]
\centering
\footnotesize
\setlength{\tabcolsep}{5pt}
\renewcommand{\arraystretch}{0.95}
\caption{Matched global state accuracy at $N=2048$ ($8189$ unknowns);
runtimes are in seconds.}
\label{tab:wp-state}
\begin{tabular}{@{}ccccc@{}}
\toprule
$\varepsilon_x$
& BDF-2-EOP
& IMP-full
& IMP-modified
& IMP-simple \\
\midrule
$1.0\times10^{-4}$
& $0.859$
& $8.630\;(10.05\times)$
& $4.152\;(4.83\times)$
& $2.735\;(3.18\times)$ \\

$3.0\times10^{-5}$
& $1.550$
& $11.587\;(7.48\times)$
& $7.336\;(4.73\times)$
& $4.182\;(2.70\times)$ \\

$1.0\times10^{-5}$
& $2.681$
& $20.317\;(7.58\times)$
& $12.274\;(4.58\times)$
& $6.266\;(2.34\times)$ \\

$3.0\times10^{-6}$
& $4.842$
& $36.634\;(7.57\times)$
& $21.852\;(4.51\times)$
& $10.075\;(2.08\times)$ \\
\bottomrule
\end{tabular}
\end{table}

\begin{table}[htb]
\centering
\footnotesize
\setlength{\tabcolsep}{5pt}
\renewcommand{\arraystretch}{0.95}
\caption{Matched Hamiltonian accuracy at $N=2048$ ($8189$ unknowns);
runtimes are in seconds.}
\label{tab:wp-energy}
\begin{tabular}{@{}ccccc@{}}
\toprule
$\varepsilon_H$
& BDF-2-EOP
& IMP-full
& IMP-modified
& IMP-simple \\
\midrule
$5.0\times10^{-5}$
& $1.539$
& $9.204\;(5.98\times)$
& $4.789\;(3.11\times)$
& $3.040\;(1.98\times)$ \\

$2.0\times10^{-5}$
& $2.429$
& $11.669\;(4.80\times)$
& $7.383\;(3.04\times)$
& $4.203\;(1.73\times)$ \\

$1.0\times10^{-5}$
& $3.432$
& $16.629\;(4.85\times)$
& $10.215\;(2.98\times)$
& $5.424\;(1.58\times)$ \\

$5.0\times10^{-6}$
& $4.806$
& $23.616\;(4.91\times)$
& $14.159\;(2.95\times)$
& $7.030\;(1.46\times)$ \\

$3.0\times10^{-6}$
& $6.159$
& $30.278\;(4.92\times)$
& $18.100\;(2.94\times)$
& $8.618\;(1.40\times)$ \\
\bottomrule
\end{tabular}
\end{table}

BDF-2-EOP is faster throughout the tested work--precision range.
Its advantage over modified and simple Newton decreases toward tighter
tolerances, where the smaller time steps make the midpoint nonlinear
problems easier. The same trend occurs for both $\varepsilon_x$ and $\varepsilon_H$.

\paragraph{Problem-size scaling.}
We next increase the system size from $253$ to $131\,069$ unknowns at
fixed $\delta t=5\times10^{-3}$. Table~\ref{tab:size-scaling} shows that
BDF-2-EOP remains faster than all midpoint realizations throughout
the tested range, while its advantage over simple Newton increases with
problem size.

\begin{table}[t]
\centering
\footnotesize
\setlength{\tabcolsep}{2.5pt}
\renewcommand{\arraystretch}{0.90}
\caption{Single-core problem-size scaling for the second-order methods
with $\delta t=5\times10^{-3}$ and $T=5$. Times are median wall-clock
times over three runs in seconds; parentheses give runtime ratios
relative to BDF-2-EOP.}
\label{tab:size-scaling}
\begin{tabular}{@{}rrcccc@{}}
\toprule
$N$ & Unknowns
& BDF-2-EOP
& IMP-full
& IMP-modified
& IMP-simple \\
\midrule
64
& 253
& 0.334
& 8.716 (26.11$\times$)
& 4.707 (14.10$\times$)
& 0.928 (2.78$\times$)
\\
128
& 509
& 0.377
& 8.982 (23.83$\times$)
& 4.911 (13.03$\times$)
& 1.106 (2.94$\times$)
\\
256
& 1\,021
& 0.459
& 9.755 (21.23$\times$)
& 5.432 (11.82$\times$)
& 1.473 (3.21$\times$)
\\
512
& 2\,045
& 0.627
& 11.184 (17.85$\times$)
& 6.375 (10.17$\times$)
& 2.180 (3.48$\times$)
\\
1\,024
& 4\,093
& 0.943
& 13.823 (14.66$\times$)
& 8.183 (8.68$\times$)
& 3.520 (3.73$\times$)
\\
2\,048
& 8\,189
& 1.558
& 19.169 (12.30$\times$)
& 11.782 (7.56$\times$)
& 6.251 (4.01$\times$)
\\
4\,096
& 16\,381
& 2.818
& 33.752 (11.98$\times$)
& 19.253 (6.83$\times$)
& 11.624 (4.12$\times$)
\\
8\,192
& 32\,765
& 5.366
& 60.475 (11.27$\times$)
& 39.011 (7.27$\times$)
& 22.690 (4.23$\times$)
\\
16\,384
& 65\,533
& 10.573
& 110.093 (10.41$\times$)
& 72.405 (6.85$\times$)
& 44.803 (4.24$\times$)
\\
32\,768
& 131\,069
& 21.168
& 213.568 (10.09$\times$)
& 133.433 (6.30$\times$)
& 90.801 (4.29$\times$)
\\
\bottomrule
\end{tabular}
\end{table}
For $N\ge8192$, the runtime grows approximately linearly with problem
size. Correspondingly, the normalized cost is nearly constant over the
three largest systems (Table~\ref{tab:normalized-scaling}), indicating a
stable large-problem computational regime for this single-core
experiment.

\begin{table}[htb]
\centering
\footnotesize
\setlength{\tabcolsep}{5pt}
\renewcommand{\arraystretch}{0.95}
\caption{Wall-clock cost per unknown and time step in the large-problem
regime, in $\mathrm{ns}/(\text{unknown}\cdot\text{step})$.}
\label{tab:normalized-scaling}
\begin{tabular}{@{}rrcccc@{}}
\toprule
$N$ & Unknowns
& BDF-2-EOP
& IMP-full
& IMP-modified
& IMP-simple \\
\midrule
8\,192  & 32\,765  & 176.5 & 2041.3 & 1328.1 & 740.0 \\
16\,384 & 65\,533  & 173.8 & 1860.9 & 1223.6 & 742.1 \\
32\,768 & 131\,069 & 172.7 & 1818.4 & 1219.3 & 743.3 \\
\bottomrule
\end{tabular}
\end{table}

\paragraph{Impact of nonlinearities.}
Finally, we vary the nonlinear coefficients according to
\[
(\beta,\beta_J,\beta_R)=\lambda(4,4,2),
\qquad
\lambda\in\{0.25,0.5,1,2\},
\]
with all other parameters fixed. Since $\beta_J$ enters the nonlinear
strain transformation, the initial state is reconstructed for each
$\lambda$ so that the physical initial-strain amplitude $a_0=0.5$
remains unchanged.

All runs use $N=2048$, $\delta t=5\times10^{-3}$, and $T=5$.
Reported times are medians of three single-core runs, and parenthetical
values denote runtime ratios relative to BDF-2-EOP.

\begin{table}[htb]
\centering
\footnotesize
\setlength{\tabcolsep}{2.5pt}
\renewcommand{\arraystretch}{0.86}
\caption{Effect of nonlinear strength on solver cost.}
\label{tab:nonlinearity-strength}
\begin{tabular}{@{}clrrrr@{}}
\toprule
$\lambda$ & Method & Time [s] & Mean it. & Max it. & Solves \\
\midrule
0.25
& BDF-2-EOP     & 1.575                  & --   & -- & 1000 \\
& IMP-full     & 23.194 (14.72$\times$) & 2.00 & 2  & 2000 \\
& IMP-modified & 14.191 (9.01$\times$)  & 2.75 & 3  & 2755 \\
& IMP-simple   & 5.031 (3.19$\times$)   & 4.57 & 5  & 4567 \\
\addlinespace[0.3ex]
0.50
& BDF-2-EOP     & 1.590                  & --   & -- & 1000 \\
& IMP-full     & 23.312 (14.66$\times$) & 2.00 & 2  & 2000 \\
& IMP-modified & 14.202 (8.93$\times$)  & 2.81 & 3  & 2813 \\
& IMP-simple   & 5.680 (3.57$\times$)   & 5.22 & 6  & 5222 \\
\addlinespace[0.3ex]
1.00
& BDF-2-EOP     & 1.607                  & --   & -- & 1000 \\
& IMP-full     & 23.211 (14.44$\times$) & 2.00 & 2  & 2000 \\
& IMP-modified & 14.357 (8.93$\times$)  & 2.91 & 3  & 2913 \\
& IMP-simple   & 6.387 (3.97$\times$)   & 5.95 & 7  & 5953 \\
\addlinespace[0.3ex]
2.00
& BDF-2-EOP     & 1.638                  & --   & -- & 1000 \\
& IMP-full     & 22.906 (13.98$\times$) & 2.00 & 2  & 2000 \\
& IMP-modified & 14.101 (8.61$\times$)  & 3.00 & 3  & 2999 \\
& IMP-simple   & 7.130 (4.35$\times$)   & 6.77 & 8  & 6771 \\
\bottomrule
\end{tabular}
\end{table}

The linear work of BDF-2-EOP is independent of $\lambda$, whereas
the Jacobian-reuse midpoint variants require progressively more
nonlinear iterations as the nonlinear coefficients increase. The effect
is strongest for simple Newton; full Newton remains at two iterations
per step because its Jacobian is updated at every iteration.
\subsubsection{Comparison with SUNDIALS IDA}
\label{sec:ida-comparison}

As a reference against a mature DAE solver, we compare BDF-2-EOP
with SUNDIALS IDA through \texttt{sundials4py}~7.8.0.

IDA is applied directly to the pH-DAE through
\[
F(t,x,\dot{x})=E\dot{x}-f(x)-Bu(t),
\qquad
J_F=c_jE-Df(x),
\]
where $c_j$ is determined internally by IDA. After a local permutation,
the Jacobian has structural half-bandwidth four; we therefore use the
native SUNDIALS banded direct solver. The interface used here did not
expose KLU. Scalar tolerances satisfy
$\mathrm{atol}=10^{-2}\cdot \mathrm{rtol}$.

The default configuration uses adaptive variable-order BDF integration
with $q_{\max}=5$. We additionally restrict IDA to $q_{\max}=2$ while
retaining adaptive step selection. The latter provides an equal-order
diagnostic for assessing the contribution of IDA's higher-order
capability, although differences in adaptivity, variable BDF
coefficients, nonlinear iteration, and linear-system setup remain.
All comparisons use $N=8192$ ($32765$ unknowns) and $T=5$.
BDF-2-EOP uses two reusable factorizations and one linear solve per
time step.

\paragraph{Matched accuracy.}
For each BDF-2-EOP operating point, we select the fastest actually
measured IDA configuration whose error does not exceed the corresponding
BDF-2 error; no interpolation of the IDA data is used. Define
\[
\rho_t:=\frac{t_{\mathrm{BDF2}}}{t_{\mathrm{IDA}}},
\]
so that $\rho_t>1$ favors IDA.

\begin{table}[htb]
\centering
\footnotesize
\setlength{\tabcolsep}{3.5pt}
\renewcommand{\arraystretch}{0.95}
\caption{Measured comparison at matched state accuracy for $N=8192$.}
\label{tab:ida-state}
\begin{tabular}{@{}cccccccc@{}}
\toprule
$\varepsilon_x^{\mathrm{BDF2}}$
& $t_{\mathrm{BDF2}}$ [s]
& \multicolumn{3}{c}{IDA, $q_{\max}=5$}
& \multicolumn{3}{c}{IDA, $q_{\max}=2$} \\
\cmidrule(lr){3-5}\cmidrule(lr){6-8}
& & rtol & $t$ [s] & $\rho_t$
  & rtol & $t$ [s] & $\rho_t$ \\
\midrule
$4.108\times10^{-4}$ & 1.0017
& $3\times10^{-4}$ & 0.8743 & 1.15
& $1\times10^{-4}$ & 1.3689 & 0.73 \\
$1.032\times10^{-4}$ & 1.9972
& $1\times10^{-4}$ & 0.9817 & 2.03
& $1\times10^{-5}$ & 2.2962 & 0.87 \\
$2.585\times10^{-5}$ & 3.9768
& $3\times10^{-5}$ & 1.1338 & 3.51
& $1\times10^{-6}$ & 4.0515 & 0.98 \\
$6.470\times10^{-6}$ & 7.9000
& $3\times10^{-6}$ & 1.4458 & 5.46
& $1\times10^{-7}$ & 7.7887 & 1.01 \\
$1.618\times10^{-6}$ & 15.7001
& $1\times10^{-6}$ & 1.6549 & 9.49
& $1\times10^{-8}$ & 15.4511 & 1.02 \\
\bottomrule
\end{tabular}
\end{table}

Unrestricted IDA becomes increasingly advantageous as the state error is
reduced, with $\rho_t$ increasing from $1.15$ to $9.49$. In contrast,
the $q_{\max}=2$ comparison is nearly balanced: BDF-2-EOP is faster
at the two coarser points, while the three finer comparisons satisfy
$0.98\le\rho_t\le1.02$. At the finest point, BDF-2-EOP takes $4000$
fixed steps, whereas order-two IDA takes $8649$ accepted steps and
$8928$ nonlinear iterations, yet their runtimes are $15.70$\,s and
$15.45$\,s, respectively.

\begin{table}[htb]
\centering
\footnotesize
\setlength{\tabcolsep}{3.5pt}
\renewcommand{\arraystretch}{0.95}
\caption{Measured comparison at matched Hamiltonian accuracy for
$N=8192$.}
\label{tab:ida-H}
\begin{tabular}{@{}cccccccc@{}}
\toprule
$\varepsilon_H^{\mathrm{BDF2}}$
& $t_{\mathrm{BDF2}}$ [s]
& \multicolumn{3}{c}{IDA, $q_{\max}=5$}
& \multicolumn{3}{c}{IDA, $q_{\max}=2$} \\
\cmidrule(lr){3-5}\cmidrule(lr){6-8}
& & rtol & $t$ [s] & $\rho_t$
  & rtol & $t$ [s] & $\rho_t$ \\
\midrule
$6.778\times10^{-4}$ & 1.0017
& $3\times10^{-4}$ & 0.8743 & 1.15
& $3\times10^{-4}$ & 1.1335 & 0.88 \\
$1.699\times10^{-4}$ & 1.9972
& $3\times10^{-4}$ & 0.8743 & 2.28
& $3\times10^{-5}$ & 1.7417 & 1.15 \\
$4.252\times10^{-5}$ & 3.9768
& $1\times10^{-4}$ & 0.9817 & 4.05
& $3\times10^{-6}$ & 3.0020 & 1.32 \\
$1.064\times10^{-5}$ & 7.9000
& $3\times10^{-5}$ & 1.1338 & 6.97
& $3\times10^{-7}$ & 5.7542 & 1.37 \\
$2.659\times10^{-6}$ & 15.7001
& $3\times10^{-6}$ & 1.4458 & 10.86
& $1\times10^{-7}$ & 7.7887 & 2.02 \\
\bottomrule
\end{tabular}
\end{table}

The Hamiltonian comparison gives the same qualitative picture.
Unrestricted IDA reaches $\rho_t=10.86$ at the finest measured
accuracy, whereas the $q_{\max}=2$ comparison remains within a factor
of $2.02$. Taken together, the two diagnostics strongly indicate that
IDA's higher-order capability is the dominant contributor to its
increasing high-accuracy advantage.

\paragraph{Structure.}
Work--precision alone does not capture the structural distinction.
Table~\ref{tab:ida-structure} reports representative diagnostics for
BDF-2-EOP at $\delta t=10^{-2}$ and unrestricted IDA at
$\mathrm{rtol}=10^{-5}$. These operating points are not accuracy
matched, so the physical balance defects are reported as diagnostics
rather than as a comparative accuracy measure.

\begin{table}[t]
\centering
\footnotesize
\setlength{\tabcolsep}{4pt}
\renewcommand{\arraystretch}{0.95}
\caption{Representative structure diagnostics at $N=8192$.}
\label{tab:ida-structure}
\begin{tabular}{@{}lcccccc@{}}
\toprule
Method
& $V_{\mathrm{aux}}$
& $V_H$
& Balance defect
& Alg. defect
& Aux. gap
& Setups \\
\midrule
BDF-2-EOP
& $0$
& $0$
& $1.936\times10^{-2}$
& $2.776\times10^{-16}$
& $0$
& 2 \\
IDA
& --
& $2.170\times10^{-7}$
& $1.405\times10^{-2}$
& $3.242\times10^{-12}$
& --
& 19 \\
\bottomrule
\end{tabular}
\end{table}

%Physical 
Passivity is observed for both methods at these settings.
BDF-2-EOP additionally satisfies the auxiliary-passivity inequality
of Lemma~\ref{lem:auxiliary-passivity} by construction. Moreover, the
EOP budget is inactive throughout this experiment, so the auxiliary
energy tracks $H(x^\ell)+C_0$ to roundoff accuracy in accordance with
Corollary~\ref{cor:exact-energy-tracking}. The algebraic residual is also
at roundoff level, and only two reusable linear-system factorizations are
required, compared with $19$ IDA linear-solver setups.

Overall, unrestricted IDA is the fastest method at high accuracy,
reaching an approximately one-order-of-magnitude advantage at the finest
measured points. Restricting IDA to order two changes the comparison
substantially: state work--precision is essentially equal in the
fine-accuracy regime, while the Hamiltonian comparison remains within a
factor of two. This supports the interpretation that the pronounced
advantage of unrestricted IDA is driven predominantly by its
higher-order adaptive integration. BDF-2-EOP instead provides a
fixed-linear, one-solve-per-step discretization together with a discrete
auxiliary-passivity guarantee.

%-----------------------
\section{Conclusions}

	We have developed an EOP-GSAV framework for nonlinear index-one
	port-Hamiltonian DAEs that separates the state-dependent nonlinear terms
	from a constant implicit core. To this end, we gathered all state dependence
	in a single explicitly evaluated term, which allowed us to show that the
	resulting BDF-1 and BDF-2 schemes require one linear solve per time step with
	reusable factorizations, while retaining order-$k$ convergence and a discrete
	auxiliary-passivity inequality. When the EOP budget is inactive, the auxiliary energy also
	tracks the shifted physical Hamiltonian exactly.

	The numerical experiments illustrate both the robustness and the
	computational implications of this construction. In the nonlinear stress
	test, BDF-2-EOP remains robust where Jacobian-reuse midpoint
	iterations deteriorate or fail. On the large pH-DAE benchmark it
	outperforms the tested implicit-midpoint realizations over the
	matched-accuracy range while avoiding repeated nonlinear solves and
	state-dependent factorizations.
	
	Comparison with SUNDIALS IDA clarifies the scope of this advantage.
	Unrestricted variable-order IDA becomes substantially faster at high
	accuracy, whereas the order-two diagnostic yields comparable state
	work--precision and remains within a small constant factor in the
	Hamiltonian metric. This strongly indicates that higher-order adaptive
	integration is the dominant contributor to IDA's high-accuracy
	advantage. BDF-2-EOP instead provides fixed linear algebra,
	predictable one-solve-per-step cost, and a passivity guarantee built
	directly into the discretization.
	
	These results show that, for the pH-DAE class considered here,
	structure-preserving nonlinear integration need not require repeated
	nonlinear solves. 
    As future work, we plan to develop higher-order and adaptive EOP-GSAV
	schemes, to treat more general port and descriptor structures, and to address
	spatially distributed port-Hamiltonian systems with boundary ports.
	The fixed-operator formulation also motivates further investigation of
	large-scale sparse and iterative linear algebra, where repeated setup
	costs and memory traffic become increasingly important. Moreover, it is expected that the implementation of the suggested GSAV method in (energy-optimal) control of strongly nonlinear (such as thermodynamic) port-Hamiltonian systems~\cite{philipp2024optimal,esterhuizen2026nonlinear} leads to highly-efficient structure-preserving nonlinear optimal control solvers.

\appendix
%--------------------------
%\section{Benchmark}
\section{Nonlinear mass-spring-damper benchmark system} \label{app:benchmark}
%--------------------------
\begin{figure}[htbp]
    \centering
    \begin{tikzpicture}[
        heavy/.style={circle, draw, thick, minimum size=10mm, fill=gray!40},
        light/.style={circle, draw, thick, minimum size=6mm, fill=gray!10},
        fpu_spring/.style={
            decoration={
                zigzag,
                pre length=0.6cm,
                post length=0.6cm,
                segment length=5pt,
                amplitude=4pt
            },
            decorate, thick
        },
        linear_spring/.style={
            decoration={
                zigzag,
                pre length=0.4cm,
                post length=0.4cm,
                segment length=4pt,
                amplitude=3pt
            },
            decorate, thick
        },
        connect/.style={thick},
        massless/.style={circle, fill=black, inner sep=1.5pt},
        input/.style={->, >={Stealth[length=3mm]}, very thick, blue},
        scale=0.8
    ]

        % ---------------------------------------------------------
        % Masses
        % ---------------------------------------------------------
        \node[heavy] (m1) at (0,0) {$m_1$};
        \node[light] (m2) at (4.5,0) {$m_2$};

        \node[heavy] (mN1) at (9.5,0) {$m_{N-1}$};
        \node[light] (mN) at (14,0) {$m_\ell$};

        \node[below=0.1cm of m1] {$q_1$};
        \node[below=0.1cm of m2] {$q_2$};
        \node[below=0.1cm of mN1] {$q_{N-1}$};
        \node[below=0.1cm of mN] {$q_\ell$};

        % ---------------------------------------------------------
        % Distributed forcing
        % ---------------------------------------------------------
        \node[blue] at (7,3.0)
        {$b_i u(t),\qquad b_i=\sin\!\left(\dfrac{2\pi i}{L_f}\right)$};

        \draw[input] (0,2.3) -- (m1.north);
        \draw[input] (4.5,2.3) -- (m2.north);
        \draw[input] (9.5,2.3) -- (mN1.north);
        \draw[input] (14,2.3) -- (mN.north);

        % ---------------------------------------------------------
        % First branch
        % ---------------------------------------------------------
        \draw[connect] (m1) -- (0,1.5);
        \draw[connect] (m1) -- (0,-1.5);
        \draw[connect] (m2) -- (4.5,1.5);
        \draw[connect] (m2) -- (4.5,-1.5);

        % Nonlinear FPU spring
        \draw[fpu_spring] (0,-1.5) -- (4.5,-1.5)
        node[midway, below=0.3cm]
        {$V_\beta(\delta_1)$};

        % Maxwell element
        \node[massless, label=above:{$q_{M,1}$}]
        (qm1) at (2.25,1.5) {};

        \draw[linear_spring] (0,1.5) -- (qm1)
        node[midway, above=0.3cm] {$K_M$};

        \draw[connect] (qm1) -- (2.75,1.5);
        \draw[connect] (2.75,1.8) -- (2.75,1.2);
        \draw[connect] (2.75,1.8) -- (3.75,1.8);
        \draw[connect] (2.75,1.2) -- (3.75,1.2);
        \draw[connect] (3.25,1.7) -- (3.25,1.3);
        \draw[connect] (3.25,1.5) -- (4.5,1.5);
        \node at (3.25,2.1) {$\eta_1$};

        % ---------------------------------------------------------
        % Continuation
        % ---------------------------------------------------------
        \draw[connect, dashed] (m2) -- (6,0);
        \draw[connect, dashed] (8,0) -- (mN1);
        \node at (7,0) {\huge $\cdots$};

        \draw[connect, dashed] (4.5,1.5) -- (6,1.5);
        \draw[connect, dashed] (8,1.5) -- (9.5,1.5);
        \node at (7,1.5) {\huge $\cdots$};

        \draw[connect, dashed] (4.5,-1.5) -- (6,-1.5);
        \draw[connect, dashed] (8,-1.5) -- (9.5,-1.5);
        \node at (7,-1.5) {\huge $\cdots$};

        % ---------------------------------------------------------
        % Last branch
        % ---------------------------------------------------------
        \draw[connect] (mN1) -- (9.5,1.5);
        \draw[connect] (mN1) -- (9.5,-1.5);
        \draw[connect] (mN) -- (14,1.5);
        \draw[connect] (mN) -- (14,-1.5);

        % Nonlinear FPU spring
        \draw[fpu_spring] (9.5,-1.5) -- (14,-1.5)
        node[midway, below=0.3cm]
        {$V_\beta(\delta_{N-1})$};

        % Maxwell element
        \node[massless, label=above:{$q_{M,N-1}$}]
        (qmN1) at (11.75,1.5) {};

        \draw[linear_spring] (9.5,1.5) -- (qmN1)
        node[midway, above=0.3cm] {$K_M$};

        \draw[connect] (qmN1) -- (12.25,1.5);
        \draw[connect] (12.25,1.8) -- (12.25,1.2);
        \draw[connect] (12.25,1.8) -- (13.25,1.8);
        \draw[connect] (12.25,1.2) -- (13.25,1.2);
        \draw[connect] (12.75,1.7) -- (12.75,1.3);
        \draw[connect] (12.75,1.5) -- (14,1.5);
        \node at (12.75,2.1) {$\eta_{N-1}$};

    \end{tikzpicture}

    \caption{
    Topology of the one-dimensional nonlinear mass--spring--damper
    chain with parallel Maxwell elements. The $N$ masses alternate
    between heavy and light values. Each neighboring pair is coupled
    by a nonlinear FPU spring with potential
    $V_\beta(\delta_i)$, where
    $\delta_i=\theta(r_i)$ is the physical strain, and by a Maxwell
    branch consisting of a linear spring of stiffness $K_M$ and a
    dashpot with coefficient $\eta_i$. The massless internal nodes
    $q_{M,i}$ give rise to the index-one algebraic constraints.
    The chain is driven by the distributed input
    $b_i u(t)$ with
    $b_i=\sin(2\pi i/L_f)$.
    }
    \label{fig:nonlinear-Maxwell}
\end{figure}
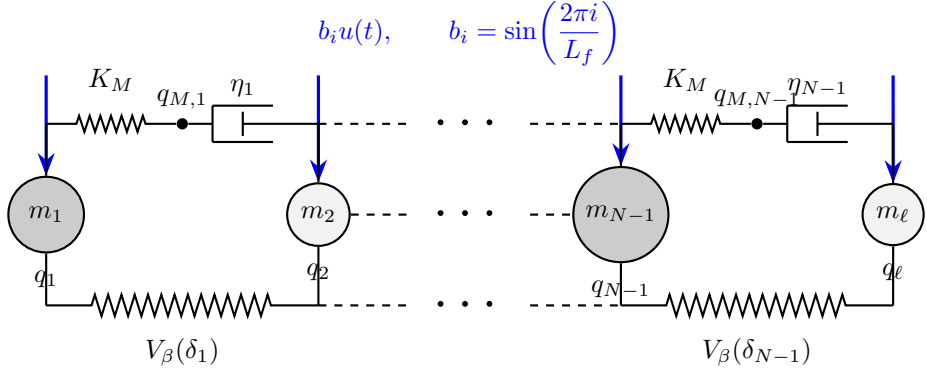

%\subsection{Nonlinear mass-spring-damper system}

We consider the one-dimensional chain shown in Fig.~\ref{fig:nonlinear-Maxwell}, consisting of $N$ alternating masses with $m_i=m_{\mathrm h}=1$ for odd $i$ and $m_i=m_{\mathrm l}<1$ for even $i$. Adjacent masses are connected by $N-1$ parallel links, each comprising a nonlinear FPU-$\beta$ spring \cite{gallavotti2007fermi} and a Maxwell branch \cite{christensen2013theory} consisting of a linear spring with stiffness $K_M$ in series with a dashpot of viscosity $\eta_i>0$. Let $q\in\R^n$ denote the physical displacements, $v=\dot q$, and $\delta_i=q_{i+1}-q_i$ the physical extension of link $i$. The FPU-$\beta$ potential is
\begin{equation*}
    V_\beta(\delta)=\frac12\delta^2+\frac{\beta}{4}\delta^4,
    \qquad
    V_\beta'(\delta)=\delta+\beta\delta^3,
    \qquad \beta>0.
\end{equation*}
To obtain a genuinely state-dependent power-preserving interconnection, we introduce the monotone generalized strain coordinate
\begin{equation*}
    r_i=\frac{1}{\sqrt{\beta_J}}\sinh\!\left(\sqrt{\beta_J}\,\delta_i\right),
    \qquad
    \delta_i=\vartheta(r_i):=
    \frac{1}{\sqrt{\beta_J}}\operatorname{arsinh}\!\left(\sqrt{\beta_J}\,r_i\right),
\end{equation*}
with the limiting case $r_i=\delta_i$ for $\beta_J=0$. Defining
\begin{equation*}
    \Gamma(r):=\operatorname{diag}\!\left(\sqrt{1+\beta_Jr_i^2}\right)_{i=1}^{N-1},
\end{equation*}
we have $\dot r=\Gamma(r)Dv$, while $\dot\delta=Dv$, so that the nonlinear interconnection results from a nonlinear change of strain coordinates and does not alter the underlying physical kinematics.

The state is
\begin{equation*}
    x=
    \begin{bmatrix}
        p^\top&r^\top&r_M^\top&z^\top
    \end{bmatrix}^\top
    \in\R^{4N-3},
\end{equation*}
where $p=Mv\in\R^n$ denotes the momentum, $r\in\R^{N-1}$ the generalized primary-spring strain, $r_M\in\R^{N-1}$ the Maxwell-spring extension, and $z\in\R^{N-1}$ the velocity of the massless internal Maxwell nodes. Here
\begin{equation*}
    M=\operatorname{diag}(m_1,\ldots,m_\ell),\qquad
    D=E_R-E_L,
\end{equation*}
where $E_L,E_R\in\R^{(N-1)\times N}$ select the left and right endpoint of each link. The Hamiltonian is
\begin{equation*}
    H(x)
    =
    \frac12p^\top M^{-1}p
    +\sum_{i=1}^{N-1}V_\beta(\delta_i)
    +\frac{K_M}{2}\|r_M\|^2,
    \qquad
    \delta=\vartheta(r),
\end{equation*}
and we choose the effort
\begin{equation*}
    e(x)=
    \begin{bmatrix}
        v\\
        \Gamma(r)^{-1}\!\left(\delta+\beta\delta^{\circ3}\right)\\
        K_Mr_M\\
        z
    \end{bmatrix},
    \qquad
    E^\top e(x)=\nabla H(x),
\end{equation*}
where $\delta^{\circ3}$ denotes componentwise cubing. In particular,
\begin{equation*}
    Q=\operatorname{diag}\!\left(M^{-1},I_{N-1},K_MI_{N-1},I_{N-1}\right),
    \qquad
    e(x)=Qx+e_{\mathrm{nl}}(x).
\end{equation*}

The descriptor and interconnection matrices are
\begin{equation*}
    E=
    \begin{bmatrix}
        I_\ell&0&0&0\\
        0&I_{N-1}&0&0\\
        0&0&I_{N-1}&0\\
        0&0&0&0
    \end{bmatrix},
    \qquad
    J(x)=
    \begin{bmatrix}
        0&-D^\top\Gamma(r)&E_L^\top&0\\
        \Gamma(r)D&0&0&0\\
        -E_L&0&0&I_{N-1}\\
        0&0&-I_{N-1}&0
    \end{bmatrix},
\end{equation*}
and hence $J(x)^\top=-J(x)$. The constant Maxwell dissipation is represented by
\begin{equation*}
    R_0=
    \begin{bmatrix}
        E_R^\top\Lambda_\eta E_R&0&0&-E_R^\top\Lambda_\eta\\
        0&0&0&0\\
        0&0&0&0\\
        -\Lambda_\eta E_R&0&0&\Lambda_\eta
    \end{bmatrix},
    \qquad
    \Lambda_\eta=\operatorname{diag}(\eta_1,\ldots,\eta_{N-1}).
\end{equation*}
To additionally exercise state-dependent dissipation, we introduce a cubic relative-velocity damper. With
\begin{equation*}
    w:=Dv,\qquad
    \Lambda_R(x):=\beta_R\operatorname{diag}(w_i^2),
    \qquad \beta_R\ge0,
\end{equation*}
we set
\begin{equation*}
    R_1(x)=
    \begin{bmatrix}
        D^\top\Lambda_R(x)D&0&0&0\\
        0&0&0&0\\
        0&0&0&0\\
        0&0&0&0
    \end{bmatrix},
    \qquad
    R(x)=R_0+R_1(x).
\end{equation*}
Thus $R(x)=R(x)^\top\ge0$ and
\begin{equation*}
    e(x)^\top R(x)e(x)
    =
    (E_Rv-z)^\top\Lambda_\eta(E_Rv-z)
    +\beta_R\sum_{i=1}^{N-1}(Dv)_i^4
    =:K(x)\ge0.
\end{equation*}

The resulting nonlinear pH-DAE is
\begin{equation*}%\label{eq:FPU-Maxwell-pH-DAE}
    E\dot x=(J(x)-R(x))e(x)+Bu,
    \qquad
    y=B^\top e(x),
\end{equation*}
or, equivalently,
\begin{align*}
    \dot p={}&-D^\top\!\left(\delta+\beta\delta^{\circ3}\right)
    +E_L^\top K_Mr_M
    -E_R^\top\Lambda_\eta(E_Rv-z)
    -D^\top\!\left(\beta_R(Dv)^{\circ3}\right)+B_pu,\\
    \dot r={}&\Gamma(r)Dv,\qquad
    \dot r_M=-E_Lv+z,\\
    0={}&-K_Mr_M+\Lambda_\eta(E_Rv-z),
\end{align*}
with $B=[B_p^\top,0,0,0]^\top$. Since $\Lambda_\eta$ is positive definite, the algebraic equation can be solved uniquely as
\begin{equation*}
    z=E_Rv-\Lambda_\eta^{-1}K_Mr_M,
\end{equation*}
and the descriptor system is index one. Moreover,
\begin{equation*}
    \frac{\mathrm d}{\mathrm dt}H(x(t))
    =
    -K(x(t))+u(t)^\top y(t),
\end{equation*}
so the continuous system is passive.

For the constant-core decomposition used by the proposed method, we take
\begin{equation*}
    J_0:=J(0),\qquad
    A:=-(J_0-R_0)Q,
\end{equation*}
and write
\begin{equation*}
    E\dot x+Ax+g(x)=Bu,\qquad
    g(x)=-(J_0-R_0)e_{\mathrm{nl}}(x)-[J(x)-J_0]e(x)+R_1(x)e(x).
\end{equation*}
For this benchmark, the nonlinear algebraic component vanishes
identically, i.e., $g_a\equiv0$.

To avoid increasingly localized dynamics as $N$ grows, the initial strain and external forcing are chosen with a fixed lattice wavelength. Let
\begin{equation*}
    \phi_i=
    \sin\!\left(\frac{2\pi i}{L_0}\right)
    +\frac14\sin\!\left(\frac{4\pi i}{L_0}+\frac{\pi}{5}\right),
    \qquad
    \widehat\phi_i=\frac{\phi_i}{\max_j|\phi_j|},
    \qquad L_0=16,
\end{equation*}
and prescribe
\begin{equation*}
    \delta_i(0)=a_0\widehat\phi_i,\qquad
    r_i(0)=\frac{1}{\sqrt{\beta_J}}
    \sinh\!\left(\sqrt{\beta_J}\,\delta_i(0)\right),
    \qquad
    p(0)=r_M(0)=z(0)=0,
\end{equation*}
for $\beta_J>0$, with $r_i(0)=\delta_i(0)$ for $\beta_J=0$. This gives a spatially extensive nonlinear initial state whose energy density remains non-vanishing as $N$ increases.

For the open-system and scaling experiments, $N$ is chosen as a multiple of $L_f$, and we use the distributed zero-mean body-force profile
\begin{equation*}
    (B_p)_i=\sin\!\left(\frac{2\pi i}{L_f}\right),
    \qquad L_f=16.
\end{equation*}
The system is driven by the globally $C^2$ compactly supported input
\begin{equation*}
    u(t)=
    \begin{cases}
        A_u\sin^4\!\left(\pi\dfrac{t-t_0}{\tau}\right),
        &t_0\le t\le t_0+\tau,\\[1mm]
        0,&\text{otherwise}.
    \end{cases}
\end{equation*}
The fixed spatial wavelengths ensure that increasing $N$ enlarges the number of dynamically active nonlinear components rather than merely appending quiescent degrees of freedom.

\bibliographystyle{plainnat}

\bibliography{references}
\end{document}